\documentclass[12pt]{article}
\usepackage[left=2cm,right=2cm,
top=2cm,bottom=2cm,bindingoffset=0cm]{geometry}
\usepackage{amssymb,amsmath}
\usepackage[utf8]{inputenc} 
\usepackage[matrix,arrow,curve]{xy}
\usepackage[final]{graphicx} 
\usepackage{mathrsfs}
\usepackage{hyperref}

\newtheorem{theorem}{Theorem}[subsection]
\newtheorem{lemma}[theorem]{Lemma}
\newtheorem{proposition}[theorem]{Proposition}
\newtheorem{remark}[theorem]{Remark}

\newtheorem{definition}[theorem]{Definition}

\newenvironment{proof}{\par $\triangleleft$}{$\triangleright$}

\begin{document}

\begin{center}

\Large \textbf{Multiplication operators on $L_p$ spaces and homological triviality of respective category of modules}\\[0.5cm]
\small {Norbert Nemesh}\\[0.5cm]

\end{center}
\thispagestyle{empty}

\begin{abstract}
We give complete characterisation of topologically injective (bounded below), topologically surjective (open mapping), isometric and coisometric (quotient mapping) multiplication operators between $L_p$ spaces defined on different $\sigma$-finite measure spaces. We prove that all such operators invertible from the left or from the right. As the consequence we prove that all objects of the category of $L_p$ spaces considered as left Banach modules over algebra of bounded measurable functions are metrically, extremelly and relatively projective, injective and flat.
\end{abstract}

\section{Preliminaries}

\subsection{Measure theoretic facts}

Let $(\Omega,\Sigma,\mu)$ be a measure space with $\sigma$-additive real valued measure. We say that $\Omega'\in\Sigma$ is an atom if $\mu(\Omega')>0$ and for every $E\in\Sigma$ such that $E\subset\Omega'$ either $\mu(E)=0$ or $\mu(\Omega'\setminus E)=0$. By $A(\Omega,\mu)$ we denote the set of atoms of $(\Omega,\Sigma,\mu)$. Now we present several standard facts from measure theory.

\begin{lemma}\label{AtomDescInSigmFinMeasSp} If $(\Omega,\Sigma,\mu)$ is a $\sigma$-finite measure space then all its atoms are of the finite measure.
\end{lemma}
\begin{proof} Since $(\Omega,\Sigma,\mu)$ is $\sigma$-finite we have representaion $\Omega=\bigcup_{n\in\mathbb{N}} F_n$ as disjoint union of measurable sets of finite measure. Assume we have $\Omega'\in A(\Omega,\mu)$ of infinite measure. For $n\in\mathbb{N}$ define $\tilde{F}_n=F_n\cap \Omega_n\in\Sigma$. Fix $n\in\mathbb{N}$, then $\tilde{F}_n\subset\Omega'$ and either $\mu(\tilde{F}_n)=0$ or $\mu(\tilde{F}_n)=+\infty$. Since $\mu(\tilde{F}_n)\leq\mu(F_n)<+\infty$ we get $\mu(\tilde{F}_n)=0$. As $n\in\mathbb{N}$ is arbitrary we get $\mu(\Omega')=\mu(\bigcup_{n\in\mathbb{N}}\tilde{F}_n)=\sum_{n\in\mathbb{N}}\mu(\tilde{F}_n)=0$. Contradiction, so $\mu(\Omega')<+\infty$.
\end{proof}

\begin{lemma}\label{PureAtomSpDecomp} Let $(\Omega,\Sigma,\mu)$ be a purely atomic measure space, then there exist pairwise disjoint family of atoms $\{\Omega_\lambda:\lambda\in\Lambda\}$ such that
$\Omega=\bigcup_{\lambda\in\Lambda}\Omega_\lambda$. If $(\Omega,\Sigma,\mu)$ is $\sigma$-finite, then the family $\{\Omega_\lambda:\lambda\in\Lambda\}$ is at most countable.
\end{lemma}
\begin{proof}
Let $\mathcal{F}=\{F\subset A(\Omega,\mu):\Omega',\Omega''\in F\implies \Omega'\cap\Omega''=\varnothing\}$. For $F',F''\in\mathcal{F}$ we take by definition $F'\leq F''$ if $F'\subset F''$. In this case $(\mathcal{F},\leq)$ is a partially ordered set in which every totally ordered set have upper bound. By Zorn's lemma we have maximal element $\tilde{F}=\{\tilde{\Omega}_\lambda:\lambda\in\Lambda\}$. Define $\Omega_0=\Omega\setminus(\bigcup_{\lambda\in\Lambda}\tilde{\Omega}_\lambda)$. If $\mu(\Omega_0)>0$ then since $\Omega$ is purely atomic there exist $\Omega_1\in A(\Omega,\mu)$ such that $\Omega_1\subset\Omega_0$. Consider $F=\tilde{F}\cup \{\Omega_1\}\in\mathcal{F}$, then $\tilde{F}\leq F$ and $\tilde{F}\neq F$. This contradicts maximality of $\tilde{F}$, hence $\mu(\Omega_0)=0$. Now take any $\lambda_0\in\Lambda$, then define
$$
\Omega_\lambda=
\begin{cases}
\tilde{\Omega}_{\lambda_0}\cup\Omega_0\quad&\text{if}\quad\lambda=\lambda_0\\
\tilde{\Omega}_\lambda            \quad&\text{if}\quad\lambda\neq\lambda_0
\end{cases}
$$
Clearly $\tilde{\Omega}_{\lambda_0}\cup\Omega_0$ is an atom disjoint from atoms $\tilde{\Omega}_\lambda$ for $\lambda\neq\lambda_0$. Hence $\{\Omega_\lambda:\lambda\in\Lambda\}$ is the desired family. 

If $(\Omega,\Sigma,\mu)$ is $\sigma$-finite we have represeentation $\Omega=\bigcup_{n\in\mathbb{N}}E_n$ as disjoint union of measurable sets of finite measure. Define $\Omega_{\lambda, n}=\Omega_\lambda\cap E_n$, then for each $n\in\mathbb{N}$ the we have $E_n=\bigcup_{\lambda\in\Lambda}\Omega_{\lambda,n}$ and $\Omega_{\lambda',n}\cap\Omega_{\lambda'',n}=\varnothing$ for $\lambda'\neq\lambda''$. Since $\mu(E_n)<+\infty$, then the family $\{\lambda\in\Lambda:\mu(\Omega_{\lambda,n})>k^{-1}\}$ is finite for every $k\in\mathbb{N}$. Thus the family $\Lambda_n=\{\lambda\in\Lambda:\mu(\Omega_{\lambda,n})>0\}$ is at most countable. Since for all $\lambda\in\Lambda$ we have a representaion $\Omega_\lambda=\bigcup_{n\in\mathbb{N}}\Omega_{\lambda,n}$ where $\mu(\Omega_\lambda)>0$ and $\Omega_{\lambda,n}\cap\Omega_{\lambda,m}=\varnothing$, then $\mu(\Omega_{\lambda,n})>0$ for some $n\in\mathbb{N}$. In other words $\Lambda=\bigcup_{n\in\mathbb{N}}\Lambda_n$, so $\Lambda$ is at most countable as union at most countable sets $\Lambda_n$.
\end{proof}

\begin{theorem}[\cite{RoyJ}, 2.1]\label{MeasSpDecomp} Let $(\Omega,\Sigma,\mu)$ be a $\sigma$-finite measure space, then there exist purely atomic measure $\mu_1:\Sigma\to[0,+\infty]$ and non atomic measure $\mu_2:\Sigma\to[0,+\infty]$ such that $\mu=\mu_1+\mu_2$ and $\mu_1\perp\mu_2$.
\end{theorem}

\begin{theorem}[\cite{SierpW}]\label{ContOfNonAtmMeas} Let $(\Omega,\Sigma,\mu)$ be nonatomic measure space. If $E\in\Sigma$  with $\mu(E)>0$, then for all $t\in[0,\mu(E)]$ there exist $F\in\Sigma$ such that $F\subset E$ and $\mu(F)=t$
\end{theorem}

\begin{theorem}[\cite{RoyJ2}, 2.1]\label{LebMeasDecomp} Let $(\Omega,\Sigma,\mu)$, $(\Omega,\Sigma,\nu)$ be $\sigma$-finite measure spaces, then there exist a measurable function $\rho_{\nu,\mu}$ a $\sigma$-finite measure $\nu_s:\Sigma\to[0,+\infty]$ and a set $\Omega_s\in\Sigma$ such that

1) $\nu=\rho_{\nu,\mu}\cdot\mu+\nu_s$

2) $\mu\perp\nu_s$ i.e.  $\mu(\Omega_s)=\nu_s(\Omega_c)=0$, where $\Omega_c=\Omega\setminus\Omega_s$
\end{theorem}

\subsection{Decompostion of $L_p$ spaces}

All linear spaces in this article are considered over field $\mathbb{C}$. By $L_0(\Omega,\mu)$ we denote the linear space of measurable functions on $\Omega$. If $p=\infty$ then we take by definition that $1/p=0$. All equalities and inequalities about $L_p$ functions are understood up to sets of measure zero. 

\begin{proposition}\label{LpSpDecomp} Let $(\Omega,\Sigma,\mu)$ be a measure space and $p\in[1,+\infty]$. Assume we have represetation $\Omega=\bigcup_{\lambda\in\Lambda}\Omega_\lambda$ where $\Omega_{\lambda'}\cap\Omega_{\lambda''}=\varnothing$ for $\lambda'\neq\lambda''$. Then the map 
$$
I_p:L_p(\Omega,\mu)\to \bigoplus\limits_p \left\{L_p(\Omega_\lambda,\mu|_{\Omega_\lambda}):\lambda\in\Lambda\right\}, f\mapsto (\lambda\mapsto f|_{\Omega_\lambda})
$$
is an isometric isomorphism.
\end{proposition}
\begin{proof}
If $f\in L_p(\Omega,\mu)$, then, of course, $f|_{\Omega_\lambda}\in L_p(\Omega_\lambda,\mu|_{\Omega_\lambda})$ for $\lambda\in\Lambda$. So $I_p$ is well defined and, obviously, it is linear. Now it is remains to prove that $I_p$ is surjective and isometric. Let $f_\lambda\in L_p(\Omega_\lambda,\mu|_{\Omega_\lambda})$ for $\lambda\in\Lambda$ then define $f(\omega)=f_\lambda(\omega)$ if $\omega\in \Omega_\lambda$. Clearly $I_p(f)_\lambda=f_\lambda$ so $I_p$ is surjective. Then for $p\in[1,+\infty)$ we have
$$
\Vert I_p(f)\Vert_{\bigoplus\limits_p \left\{L_p(\Omega_\lambda,\mu|_{\Omega_\lambda}):\lambda\in\Lambda\right\}}
=\left(\sum\limits_{\lambda\in\Lambda}\int_{\Omega_\lambda}|f|_{\Omega_\lambda}(\omega)|^p d\mu|_{\Omega_\lambda}(\omega)\right)^{1/p}
=\left(\int_{\Omega}|f(\omega)|^pd\mu(\omega)\right)^{1/p}
=\Vert f\Vert_{L_p(\Omega,\mu)}
$$
Similarly for $p=\infty$ we have
$$
\Vert I_\infty(f)\Vert_{\bigoplus\limits_p \left\{L_p(\Omega_\lambda,\mu|_{\Omega_\lambda}):\lambda\in\Lambda\right\}}
=\sup\limits_{\lambda\in\Lambda}\mathop{\operatorname{essup}}_{\omega\in\Omega_\lambda}|f|_{\Omega_\lambda}(\omega)|
=\mathop{\operatorname{essup}}_{\omega\in\Omega}|f(\omega)|
=\Vert f\Vert_{L_\infty(\Omega,\mu)}
$$
In both cases $I_p$ is isometric.
\end{proof}

\begin{lemma}\label{FuncDescOnAtom} Let $(\Omega,\Sigma,\mu)$ be a $\sigma$-finite measure space with atom $\Omega'$. Then

1) If $p\in[1,+\infty]$ and $f\in L_p(\Omega',\mu|_{\Omega'})$, then 
$$
f(\omega)=\mu(\Omega')^{-1}\int_{\Omega'} f(\omega')d\mu(\omega')
$$
for $\omega\in\Omega'$.

2) If $p\in[1,+\infty]$ the map 
$$
J_p:L_p(\Omega',\mu|_{\Omega'})\to \ell_p(\{1\}),f\mapsto\left(1\mapsto \mu(\Omega')^{1/p-1}\int_{\Omega'} f(\omega')d\mu(\omega')\right)
$$
is an isometric isomorphism.

\end{lemma}
\begin{proof}
Since $(\Omega,\Sigma,\mu)$ is $\sigma$-finite, by lemma \ref{AtomDescInSigmFinMeasSp} we have $\mu(\Omega')<+\infty$.

1) Since $\mu(\Omega')<+\infty$, then $f\in L_p(\Omega',\mu|_{\Omega'})\subset L_1(\Omega',\mu|_{\Omega'})$. For the beginning assume that $f$ is a real valued function. Denote $k=\mu(\Omega')^{-1}\int_{\Omega'} f(\omega')d\mu(\omega')$, then consider set $S_-=f^{-1}((k,+\infty])$. Since $\Omega'$ is an atom then $\mu(S_-)=\mu(\Omega')$ or $\mu(S_-)=0$. In the first case we get
$$
\int_{\Omega'} f(\omega')\mu(\omega')
=\int_{S_-} f(\omega')\mu(\omega')
<\int_{S_-} c\mu(\omega')
=k\mu(S_-)
=k\mu(\Omega')
=\int_{\Omega'} f(\omega')\mu(\omega')
$$
Contradiction, hence $\mu(S_-)=0$. Similarly we get that $\mu(S_+)=0$ for $S_+=f^{-1}([-\infty,k))$. Hence $f(\omega)=k$ for $\mu$-almost all $\omega\in\Omega'$. If $f$ is complex valued we apply previous result to $\operatorname{Re}(f),\operatorname{Im}(f)\in L_1(\Omega',\mu|_{\Omega'})$ and get that
$$
f(\omega)
=\operatorname{Re}(f)(\omega)+i\operatorname{Im}(f)(\omega)
$$
$$
=\mu(\Omega')^{-1}\int_{\Omega'} \operatorname{Re}(f)(\omega')d\mu(\omega')+i\mu(\Omega')^{-1}\int_{\Omega'} \operatorname{Im}(f)(\omega')d\mu(\omega')
=\mu(\Omega')^{-1}\int_{\Omega'} f(\omega')d\mu(\omega')
$$
for $\mu$-almost all $\omega\in\Omega$

2) Obviously $J_p$ is linear. Take any $z\in\mathbb{C}$ and consider function $f=z\mu(\Omega_1)^{-1/p}\chi_{\Omega'}$, then $J_p(f)(1)=z$. Thus $J_p$ is surjective. Now for $p\in[1,+\infty]$ and all $f\in L_p(\Omega',\mu|_{\Omega'})$ we have
$$
\Vert J_p(f)\Vert_p
=\left|\mu(\Omega')^{1/p-1}\int_{\Omega'}f(\omega')d\mu(\omega')\right|
=\left|\mu(\Omega')^{1/p-1}k\mu(\Omega')\right|
=\mu(\Omega')^{1/p}|k|
=\Vert f\Vert_{L_p(\Omega',\mu|_{\Omega'})}
$$
Thus, the map $J_p$ is a surjective isometry, hence isometric isomorphism.
\end{proof}

\begin{proposition}\label{DescOfLpSpOnPureAtomMeasSp} Let $(\Omega,\Sigma,\mu)$ be a $\sigma$-finite purely atomic measure space, then for $p\in[1,+\infty]$, the map
$$
\tilde{I}_p:L_p(\Omega,\mu)\to \ell_p(\Lambda):f\mapsto\left (\lambda\mapsto J_p(f|_{\Omega_\lambda})(1)\right)
$$
is an isometric isomorphism. Here $\{\Omega_\lambda:\lambda\in\Lambda\}\subset A(\Omega,\mu)$ is at most countable family of pairwise disjoint atoms such that $\Omega=\bigcup_{\lambda\in\Lambda}\Omega_\lambda$.
\end{proposition}
\begin{proof}
By lemma \ref{PureAtomSpDecomp} we have a family $\{\Omega_\lambda:\lambda\in\Lambda\}\subset A(\Omega,\mu)$ of pairwise disjoint atoms whose union is $\Omega$. By proposition \ref{LpSpDecomp} we have an isometric isomorphism 
$$
L_p(\Omega,\mu)
\cong_1 \bigoplus\limits_p \left\{L_p(\Omega_\lambda,\mu|_{\Omega_\lambda}):\lambda\in\Lambda\right\}
$$
via the map $I_p(f)=\oplus_p \{f|_{\Omega_\lambda}:\lambda\in\Lambda\}$. By lemma \ref{FuncDescOnAtom} we know that $L_p(\Omega_\lambda,\mu|_{\Omega_\lambda})\cong_1\ell_p(\{1\})$ via the map $J_p$. So we get
$$
L_p(\Omega,\mu)
\cong_1 \bigoplus\limits_p \left\{L_p(\Omega_\lambda,\mu|_{\Omega_\lambda}):\lambda\in\Lambda\right\}
\cong_1 \bigoplus\limits_p \left\{\ell_p(\{1\}):\lambda\in\Lambda\right\}
=\ell_p(\Lambda)
$$
via the map $\tilde{I}_p$. Since $(\Omega,\Sigma,\mu)$ is $\sigma$-finite by lemma \ref{PureAtomSpDecomp} we get that $\Lambda$ is at most countable.
\end{proof}

\begin{proposition}\label{DescOfLpSpOnMeasSp} Let $p\in[1,+\infty]$. Let $(\Omega,\Sigma,\mu)$ be $\sigma$-finite measure space, then there exist at most countable family of atoms $\{\Omega_\lambda:\lambda\in\Lambda\}$ and a set $\Omega_{na}=\Omega\setminus\left(\bigcup_{\lambda\in\Lambda}\Omega_\lambda\right)=\Omega\setminus\Omega_{a}$ such that the map
$$
\hat{I}_p:L_p(\Omega,\mu)\to L_p(\Omega_{na},\mu|_{\Omega_{na}})\bigoplus_{p}L_p(\Omega_{a},\mu|_{\Omega_{a}}),f\mapsto (f|_{\Omega_{na}},f|_{\Omega_{a}})
$$
is isometric isomorphism and $\mu|_{\Omega_{na}}$ is nonatomic. 
\end{proposition}
\begin{proof} By theorem \ref{MeasSpDecomp} we have mutually singular purely atomic measure $\mu_1$ and nonatomic measure $\mu_2$ whose sum is $\mu$. Since they are singular there exist $\Omega_a\in\Sigma$ such that $\mu_2(\Omega_a)=\mu_1(\Omega\setminus\Omega_a)=0$. Thus $\mu|_{\Omega_a}=\mu_1$ is purely atomic and $\mu|_{\Omega_{na}}=\mu_2$ is nonatomic. Here $\Omega_{na}=\Omega\setminus\Omega_a$. By proposition \ref{LpSpDecomp} 
$$
\hat{I}_p:L_p(\Omega,\mu)\to L_p(\Omega_{na},\mu|_{\Omega_{na}})\bigoplus_{p}L_p(\Omega_{a},\mu|_{\Omega_{a}}),f\mapsto (f|_{\Omega_{na}},f|_{\Omega_{a}})
$$
is an isometric isomorphism.
\end{proof}

\begin{proposition}\label{ChngOfDenst} Let $(\Omega,\Sigma,\mu)$ be a $\sigma$-finite measure space. Let $\rho\in L_0(\Omega,\mu)$ be a  positive function. Then
$$
\bar{I}_p:L_p(\Omega,\mu)\to L_p(\Omega,\rho\cdot \mu), f\mapsto\rho^{-1/p}\cdot f
$$
is an isometric isomorphism for all $p\in[1,+\infty]$.
\end{proposition}
\begin{proof} Obviously $\bar{I}_p$ is linear. Let $p\in[1,+\infty)$, then for all $f\in L_p(\Omega,\mu)$ we have
$$
\Vert \bar{I}_p(f)\Vert_{L_p(\Omega,\rho\cdot\mu)}
=\left(\int_{\Omega}|\rho^{-1/p}(\omega)f(\omega)|^p\rho(\omega)d\mu(\omega) \right)^{1/p}
=\left(\int_{\Omega}|f(\omega)|^pd\mu(\omega) \right)^{1/p}
=\Vert f\Vert_{L_p(\Omega,\mu)}
$$
so $\bar{I}_p$ is an isometry. Now for arbitrary $f\in L_p(\Omega,\rho\cdot\mu)$ consider $h=\rho^{1/p}\cdot f$, then
$$
\Vert h\Vert_{L_p(\Omega,\mu)}
=\left(\int_{\Omega}|\rho^{1/p}(\omega)f(\omega)|^pd\mu(\omega) \right)^{1/p}
=\left(\int_{\Omega}|f(\omega)|^p\rho(\omega)d\mu(\omega) \right)^{1/p}
=\Vert f\Vert_{L_p(\Omega,\rho\cdot\mu)}
$$
So $h\in L_p(\Omega,\mu)$ and obviously $\bar{I}_p(h)=f$. Thus $\bar{I}_p$
 is a surjective isometry, i.e.  isometric isomorphism. Now let $p=+\infty$, then $\bar{I}_\infty$ is an identity map. Hence it is enough to show it is isometric. Let $E\in\Sigma$ with $\mu(E)=0$, then $(\rho\cdot\mu)(E)=\int_E\rho(\omega)d\mu(\omega)=0$. On the other hand if $(\rho\cdot\mu)(E)=\int_E\rho(\omega)d\mu(\omega)=0$ then from positivity of $\rho$ it follows that $\mu(E)=0$. So for all $f\in L_\infty(\Omega,\mu)$ we have
$$
\Vert\bar{I}_\infty(f)\Vert_{L_\infty(\Omega,\rho\cdot\mu)}
=\inf\{C>0:(\rho\cdot\mu)(|f|^{-1}((C,+\infty)))=0\}
$$
$$
=\inf\{C>0:\mu(|f|^{-1}((C,+\infty)))=0\}
=\Vert f\Vert_{L_\infty(\Omega,\mu)}
$$
Thus, $\bar{I}_\infty$ is an isometry.
\end{proof}

\begin{proposition}\label{FinDimlpEquivNorms} Let $\Lambda$ be a finite set and $p,q\in[1,+\infty]$, then there exist $C_{p,q}$ such that $\Vert x\Vert_{\ell_p(\Lambda)}\leq C_{p,q}\Vert x\Vert_{\ell_q(\Lambda)}$ for all $x\in\mathbb{C}^\Lambda$.
\end{proposition}
\begin{proof} Since $(\mathbb{C}^\Lambda,\Vert\cdot\Vert_{\ell_r(\Lambda)})$ is a normed space of finite dimension equal to $\operatorname{Card}(\Lambda)$ for every $r\in[1,+\infty]$, then all norms $\left\{\Vert\cdot\Vert_{\ell_r(\Lambda)}:r\in[1,+\infty]\right\}$ on $\mathbb{C}^\Lambda$ are equivalent. So we get the desired inequality.
\end{proof}

\subsection{Mutiplication operators}

Let $(\Omega,\Sigma,\mu)$ and $(\Omega,\Sigma,\nu)$ be two measure spaces with the same $\sigma$-algebra of measurable sets. For a given $g\in L_0(\Omega,\mu)$ and $p,q\in[1,+\infty]$ we define the multiplication operator
$$
M_g:L_p(\Omega,\mu)\to L_q(\Omega,\nu), f\mapsto g\cdot f
$$ 
Of course certain restrictions on $g$, $\mu$ and $\nu$ are required for $M_g$ to be well defined. For a given $E\in\Sigma$ by $M_g^E$ we will denote the linear operator
$$
M_g^E:L_p(E,\mu|_E)\to L_q(E,\nu|_E),f\mapsto g|_E\cdot f
$$
It is well defined because $f|_{\Omega\setminus E}=0$ implies $M_g(f)|_{\Omega\setminus E}=0$. 

\begin{proposition}\label{MultpOpSurjInjDesc} Let $(\Omega,\Sigma,\mu)$ be a measure space and $g\in L_0(\Omega,\mu)$. Denote $Z_g=g^{-1}(\{0\})$, then for the operator $M_g:L_p(\Omega,\mu)\to L_q(\Omega,\mu)$ we have

1) $\operatorname{Ker}(M_g)=\{f\in L_p(\Omega,\mu):f|_{\Omega\setminus {Z_g}}=0\}$, so $M_g$ is injective if and only if $\mu(Z_g)=0$

2) $\operatorname{Im}(M_g)\subset\{h\in L_q(\Omega,\mu): h|_{Z_g}=0\}$,so if $M_g$ is surjective then $\mu(Z_g)=0$.

\end{proposition}
\begin{proof}
1) We have the following chain of equivalences
$$
f\in\operatorname{Ker}(M_g)
\Longleftrightarrow g\cdot f=0
\Longleftrightarrow f|_{\Omega\setminus Z_g}=0
$$
And we get the desired equality. 

2) Since $g|_{Z_g}=0$ then for all $f\in L_p(\Omega,\mu)$ we have $M_g(f)|_{Z_g}=(g\cdot f)|_{Z_g}=0$, thus we get the inclusion. If $M_g$ is surjective then, clearly,  $\mu(Z_g)=0$.
\end{proof}

We want to classify multiplication operators according to following definitions

\begin{definition}\label{DefNorOpType} Let $ T:E\to F$ be bounded linear operator between normed spaces $E$ and $F$, then $ T$ is called
\newline
1) \textit{$c$-topologically injective}, if there exist $c > 0$ such that for all $x \in E$ holds $\Vert  T(x)\Vert_F\geq c\Vert x\Vert_E$. 
\newline
2) isometric, if $\Vert T(x)\Vert_F=\Vert x\Vert_E$ for all $x\in E$.
\newline
3) \textit{(strictly) $c$-topologically surjective}, if for all $c'>c$ and  $y\in F$ there exist $x \in E$ such that $ T(x) = y$ and $\Vert x \Vert_E < c' \Vert y \Vert_F$ ($\Vert x \Vert_E \leq c \Vert y \Vert_F$). 
\newline
4) (strictly) coisometric, if it is contractive and (strictly) $1$-topologically surjective.
\end{definition}

\begin{remark} If the constant $c$ is out of interest then we will simply say that operator is topologically injective or topologically surjective. In this case also there is no difference between topologically surjective and strictly topologically surjective operators.
\end{remark}

For a given $E\subset \Omega$ and $f\in L_0(E,\mu|_{E})$ by $\tilde{f}$ we will denote the function such that $\tilde{f}(\omega)=f(\omega)$ if $\omega\in E$ and $\tilde{f}(\omega)=0$ otherwise.

\begin{proposition}\label{MultOpDecompDecomp} Let $(\Omega,\Sigma,\mu)$, $(\Omega,\Sigma,\nu)$ be measure spaces and $p,q\in[1,+\infty]$. Assume we have a represetation $\Omega=\bigcup_{\lambda\in\Lambda}\Omega_\lambda$ of $\Omega$ as finite disjoint union of measurable sets. Then 

1) operator $M_g$ be $c$-topologically injective for some $c>0$ if and only if operators $M_g^{\Omega_\lambda}$ are $c'$-topologically injective for all $\lambda\in\Lambda$ and some $c'>0$

2) operator $M_g$ is $c$-topologically surjective for some $c>0$ if and only if operators $M_g^{\Omega_\lambda}$ are $c'$-topologically surjective for all $\lambda\in\Lambda$ and some $c'>0$

3) if operator $M_g$ is isometric then so are $M_g^{\Omega_\lambda}$ for all $\lambda\in\Lambda$

4) if operator $M_g$ is coisometric then so are $M_g^{\Omega_\lambda}$ for all $\lambda\in\Lambda$

\end{proposition}
\begin{proof}
1) Let $M_g$ is $c$-topologically injective. Fix $\lambda\in\Lambda$ and $f\in L_p(\Omega_\lambda,\mu|_{\Omega_\lambda})$, then 
$$
\Vert M_g^{\Omega_\lambda}(f)\Vert_{L_q(\Omega_\lambda,\nu|_{\Omega_\lambda})}
=\Vert g\cdot \tilde{f}\Vert_{L_q(\Omega,\nu)}
\geq c\Vert\tilde{f}\Vert_{L_p(\Omega,\mu)}
=c\Vert f\Vert_{L_p(\Omega_\lambda,\mu|_{\Omega_\lambda})}
$$
So $M_g^{\Omega_\lambda}$ is $c$-topologically injective for all $\lambda\in\Lambda$. 

Conversely, assume operators $\{M_g^{\Omega_\Lambda}:\lambda\in\Lambda\}$ are $c'$-topologically injective. Let $f\in L_p(\Omega,\mu)$. Using propositions \ref{LpSpDecomp}, \ref{FinDimlpEquivNorms} we get
$$
\Vert M_g(f)\Vert_{L_q(\Omega,\nu)}
=\left\Vert\left(\Vert M_g^{\Omega_\lambda}(f|_{\Omega_\lambda})\Vert_{L_q(\Omega_\lambda,\nu|_{\Omega_\lambda})}:\lambda\in\Lambda\right)\right\Vert_{\ell_q(\Lambda)}
\geq c'\left\Vert\left(\Vert f|_{\Omega_\lambda}\Vert_{L_p(\Omega_\lambda,\mu|_{\Omega_\lambda})}:\lambda\in\Lambda\right)\right\Vert_{\ell_q(\Lambda)}
$$
$$
\geq c' C_{p,q}^{-1}\left\Vert\left(\Vert f|_{\Omega_\lambda}\Vert_{L_p(\Omega_\lambda,\mu|_{\Omega_\lambda})}:\lambda\in\Lambda\right)\right\Vert_{\ell_p(\Lambda)}
=c'C_{p,q}^{-1}\Vert f\Vert_{L_p(\Omega,\mu)}
$$
Since $f$ is arbitrary $M_g$ is $c$-topologically injective for $c=c'C_{p,q}^{-1}>0$.

2) Let $M_g$ is $c$-topologically surjective. Fix $\lambda\in\Lambda$ and $h\in L_q(\Omega_\lambda,\nu|_{\Omega_\lambda})$. Then there exist $\tilde{f}\in L_p(\Omega,\mu)$ such that $M_g(\tilde{f})=\tilde{h}$ and $\Vert \tilde{f}\Vert_{L_p(\Omega,\mu)}\leq c\Vert \tilde{h}\Vert_{L_q(\Omega,\nu)}$. Consider $f=\tilde{f}|_{\Omega_\lambda}$, then $M_g^{\Omega_\lambda}(f)=\tilde{h}|_{\Omega_\lambda}=h$ and $\Vert f\Vert_{L_p(\Omega_\lambda,\mu|_{\Omega_\lambda})}\leq \Vert \tilde{f}\Vert_{L_p(\Omega,\mu)}\leq c\Vert\tilde{h}\Vert_{L_q(\Omega,\nu)}=c\Vert h\Vert_{L_q(\Omega_\lambda,\nu|_{\Omega_\lambda})}$. Since $h$  is arbitrary then $M_g^{\Omega_\lambda}$ is $c$-topologically surjective for all $\lambda\in\Lambda$.

Conversely, assume operators $\{M_g^{\Omega_\Lambda}:\lambda\in\Lambda\}$ are $c'$-topologically surjective. Let $h\in L_q(\Omega,\nu)$. From assumption for each $\lambda\in\Lambda$ we have $f_\lambda\in L_p(\Omega_\lambda,\mu|_{\Omega_\lambda})$ such that $M_g^{\Omega_\lambda}(f_\lambda)=h|_{\Omega_\lambda}$ and $\Vert f_\lambda\Vert_{L_p(\Omega_\lambda,\mu|_{\Omega_\lambda})}\leq c'\Vert h|_{\Omega_\lambda}\Vert_{L_q(\Omega_\lambda,\nu|_{\Omega_\lambda})}$. Define $f\in L_0(\Omega,\mu)$ such that $f(\omega)=f_\lambda(\omega)$ if $\omega\in\Omega_\lambda$.  Using propositions \ref{LpSpDecomp}, \ref{FinDimlpEquivNorms} we get
$$
\Vert f\Vert_{L_p(\Omega,\mu)}
=\left\Vert\left(\Vert f_\lambda\Vert_{L_p(\Omega_\lambda,\mu|_{\Omega_\lambda})}:\lambda\in\Lambda\right)\right\Vert_{\ell_p(\Lambda)}
\leq c'\left\Vert\left(\Vert h|_{\Omega_\lambda}\Vert_{L_q(\Omega_\lambda,\nu|_{\Omega_\lambda})}:\lambda\in\Lambda\right)\right\Vert_{\ell_p(\Lambda)}
$$
$$
\leq c'C_{p,q}\left\Vert\left(\Vert h|_{\Omega_\lambda}\Vert_{L_q(\Omega_\lambda,\nu|_{\Omega_\lambda})}:\lambda\in\Lambda\right)\right\Vert_{\ell_q(\Lambda)}
=c'C_{p,q}\Vert h\Vert_{L_q(\Omega,\nu)}
$$
Obviously, $M_g(f)=h$. Since $h$ is arbitrary we get that $M_g$ is $c$-topologiclly surjective for $c=c'C_{p,q}>0$.

3) Fix $\lambda\in\Lambda$ and $f\in L_p(\Omega_\lambda,\mu|_{\Omega_\lambda})$, then 
$$
\Vert M_g^{\Omega_\lambda}(f)\Vert_{L_q(\Omega_\lambda,\nu|_{\Omega_\lambda})}
=\Vert g\cdot \tilde{f}\Vert_{L_q(\Omega,\nu)}
=\Vert\tilde{f}\Vert_{L_p(\Omega,\mu)}
=\Vert f\Vert_{L_p(\Omega_\lambda,\mu|_{\Omega_\lambda})}
$$
So $M_g^{\Omega_\lambda}$ is isometric for all $\lambda\in\Lambda$

4) Fix $\lambda\in\Lambda$. Since $M_g$ is coisometric it is $1$-topologically surjective and contractive. So from paragraph 2) we see that $M_g^{\Omega_\lambda}$ is $1$-topologically surjective. Let $f\in L_p(\Omega_\lambda,\mu|_{\Omega_\lambda})$. Since $M_g$ is contractive we get
$$
\Vert M_g^{\Omega_\lambda}(f)\Vert_{L_q(\Omega_\lambda,\nu|_{\Omega_\lambda})}
=\Vert M_g(\tilde{f})\chi_{\Omega_\lambda}\Vert_{L_q(\Omega,\nu)}
=\Vert M_g(\tilde{f}\chi_{\Omega_\lambda})\Vert_{L_q(\Omega,\nu)}
\leq \Vert\tilde{f}\chi_{\Omega_\lambda}\Vert_{L_p(\Omega,\mu)}
=\Vert f\Vert_{L_p(\Omega_{\lambda},\mu|_{\Omega_\lambda})}
$$
Since $M_g^{\Omega_\lambda}$ is contractive and $1$-topologically injective it is coisometric.
\end{proof}

\begin{proposition}\label{MultOpCharacBtwnTwoSingMeasSp} Let $(\Omega,\Sigma,\mu)$ and $(\Omega,\Sigma,\nu)$ be two $\sigma$-finite measure spaces. Let $p,q\in[1,+\infty]$ and $g\in L_0(\Omega,\mu)$. If $\mu\perp\nu$ then $M_g:L_p(\Omega,\mu)\to L_q(\Omega,\nu_s)$ is zero operator.
\end{proposition}
\begin{proof} Since $\mu\perp\nu$, then there exist $\Omega_s\in\Sigma$ such that $\mu(\Omega_s)=\nu(\Omega_c)=0$, where $\Omega_c=\Omega\setminus\Omega_a$. Since $\mu(\Omega_s)=0$, then $\chi_{\Omega_c}=\chi_{\Omega}$ in $L_p(\Omega,\mu)$ and $\chi_{\Omega_c}=0$ in $L_q(\Omega,\nu)$. Now for all $f\in L_p(\Omega,\mu)$ we have $M_g(f)=M_g(f\cdot \chi_{\Omega})=M_g(f\cdot \chi_{\Omega_c})=g\cdot f\cdot\chi_{\Omega_c}=0$. Since $f$ is arbitrary $M_g=0$.
\end{proof}

\section{Properties of multiplication operators}

\subsection{Topologically injective operators}

This is the main section for subsequent study of multiplication operators. In the end we will show that topologically injective multiplication operators are isomorphisms or invertible from the left.

\begin{proposition}\label{MulpOpPropIfPeqqualsQ} Let $(\Omega,\Sigma,\mu)$ be a measure space and $g\in L_0(\Omega,\mu)$. Let $p=q$, then 

1) $M_g\in\mathcal{B}(L_p(\Omega,\mu))$ if and only if $g\in L_\infty(\Omega,\mu)$. 

2) $M_g$ is an isomorphism if and only if $C\geq |g|\geq c$ for some $C,c>0$.
\end{proposition}
\begin{proof}
1) Assume $M_g\in\mathcal{B}(L_p(\Omega,\mu))$. Assume there exist $E\in\Sigma$ with $\mu(E)>0$ such that $|g|_E|>\Vert M_g\Vert$, then
$$
\Vert M_g(\chi_E)\Vert_{L_p(\Omega,\mu)}
=\Vert g\cdot\chi_E\Vert_{L_p(\Omega,\mu)}
>\Vert M_g\Vert\Vert\chi_E\Vert_{L_p(\Omega,\mu)}
$$
Contradiction, hence for all $E\in\Sigma$ with $\mu(E)>0$ we have $|g|_E|\leq \Vert M_g\Vert$ i.e.  $|g|\leq \Vert M_g\Vert$. Thus $g\in L_\infty(\Omega,\mu)$

Conversely. let $g\in L_\infty(\Omega,\mu)$ then $|g|\leq C$ for some $C>0$. Now for any $p\in[1,+\infty]$ we have and all $f\in L_p(\Omega,\mu)$ we have
$$
\Vert M_g(f)\Vert_{L_p(\Omega,\mu)}
=\Vert g\cdot f\Vert_{L_p(\Omega,\mu)}
\leq C\Vert f\Vert_{L_p(\Omega,\mu)}
$$
Hence $M_g\in\mathcal{B}(L_p(\Omega,\mu))$

2) Note that $M_g^{-1}=M_{1/g}$ as linear maps provided $g$ is invertible. Now $M_g$ is an isomorphism if and only if $M_g$ and $M_g^{-1}$ are bounded operators. From previous paragraph and equality $M_g^{-1}=M_{1/g}$ we see that it is equivalent to boundedness of $g$ and $1/g$. This is equivalent to $C\geq|g|\geq c$ for some $C,c>0$
.
\end{proof}

\begin{proposition}\label{EquivMultOp} Let $(\Omega,\Sigma,\mu)$ be a $\sigma$-finite purely atomic measure space. Let $p,q\in[1,+\infty]$ and $g\in L_0(\Omega,\mu)$, then the operator $\tilde{M}_{\tilde{g}}:=\tilde{I}_q M_g\tilde{I}_p^{-1}\in\mathcal{B}(\ell_p(\Lambda),\ell_q(\Lambda))$ is a multiplication operator by the function $\tilde{g}:\Lambda\to\mathbb{C},\lambda\mapsto \mu(\Omega_\lambda)^{1/q-1/p-1}\int_{\Omega_\lambda}f(\omega)d\mu(\omega)$ where $\{\Omega_\lambda:\lambda\in\Lambda\}$ is at most countable decomposition of $\Omega$ into pairwise disjoint atoms guranteed by proposition \ref{DescOfLpSpOnPureAtomMeasSp}.
\end{proposition}
\begin{proof} Let $p,q\in[1,+\infty]$. For any $x\in\ell_p(\Lambda)$ we have
$$
\tilde{M}_{\tilde{g}}(x)(\lambda)
=(\tilde{I}_q((M_g\tilde{I}_p^{-1})(x))(\lambda)
=J_q(M_g(\tilde{I}_p^{-1}(x))|_{\Omega_\lambda})(1)
$$
$$
=J_q((g\cdot\tilde{I}_p^{-1}(x))|_{\Omega_\lambda})(1)
=\mu(\Omega_\lambda)^{1/q-1}\int_{\Omega_\lambda}(g|_{\Omega_\lambda}\cdot \tilde{I}_p^{-1}(x)|_{\Omega_\lambda})(\omega)d\mu(\omega)
$$
$$
=\mu(\Omega_\lambda)^{1/q-1}\int_{\Omega_\lambda}(g\cdot \mu(\Omega)^{-1/p}x(\lambda)\chi_{\Omega_{\lambda}})(\omega)d\mu(\omega)
=x(\lambda)\mu(\Omega_\lambda)^{1/q-1/p-1}\int_{\Omega_\lambda} g(\omega)d\mu(\omega)
$$
Thus $\tilde{M}_{\tilde{g}}$ is a multiplication operator where $\tilde{g}(\lambda)=\mu(\Omega_\lambda)^{1/q-1/p-1}\int_{\Omega_\lambda} g(\omega)d\mu(\omega)$
\end{proof}

Since $\tilde{I}_p$ and $\tilde{I}_q$ are isometric isomorphisms then $M_g$ is topologically injective if and only if $\tilde{M}_{\tilde{g}}$ is topologically injective. 

\begin{proposition}\label{TopInjMultOpCharacOnPureAtomMeasSp} Let $(\Omega,\Sigma,\mu)$ be $\sigma$-finite purely atomic measure space, $p,q\in[1,+\infty]$ and $g\in L_0(\Omega,\mu)$ then the following are equivalent

i) $M_g\in\mathcal{B}(L_p(\Omega,\mu),L_q(\Omega,\mu))$ is topologically injective 

ii) $|g|\geq c$ for some $c>0$ and if $p\neq q$ the space $(\Omega,\Sigma,\mu)$ have finitely many atoms.
\end{proposition}
\begin{proof}
$i)\implies ii)$ Assume $M_g$ is topologically injective, then so is $\tilde{M}_{\tilde{g}}$, i.e. $\Vert\tilde{M}_{\tilde{g}}(x)\Vert_{\ell_q(\Lambda)}\geq c'\Vert x\Vert_{\ell_p(\Lambda)}$ for all $x\in\ell_p(\Lambda)$ and some $c'>0$. Now we use at most countable decomposition $\{\Omega_\lambda:\lambda\in\Lambda\}$ of $\Omega$ into pairwise disjoint atoms of $\Omega$ given by proposition \ref{DescOfLpSpOnPureAtomMeasSp}. Now we will consider two big cases

1) Let $p\neq q$. Assume $\Lambda$ is countable. 

1.1) Consider subcase $p,q<+\infty$. If $\Lambda$ is countable, then we get a  contradiction, because by Pitt's theorem (\cite{AlbiacKalton}, 2.1.6) there is no embeddings between $\ell_p(\Lambda)$ and $\ell_q(\Lambda)$ spaces for countable $\Lambda$ and $p,q\in[1,+\infty)$, $p\neq q$. 

1.2) Consider subcase $q=+\infty$. Take any finite family $F\subset\Lambda$, then 
$$
\sup_{\lambda\in\Lambda}|\tilde{g}(\lambda)|
\geq\max_{\lambda\in F}|\tilde{g}(\lambda)|
=\left\Vert\tilde{M}_{\tilde{g}}\left(\sum_{\lambda\in F}e_\lambda\right)\right\Vert_{\ell_\infty(\Lambda)}
\geq c'\left\Vert\sum_{\lambda\in F}e_\lambda\right\Vert_{\ell_p(\Lambda)}
=c'\operatorname{Card}(F)
$$
Since $\Lambda$ is countable $\sup_{\lambda\in\Lambda}|\tilde{g}(\lambda)|\geq c'\sup_{F\subset\Lambda}\operatorname{Card}(F)=+\infty$. On the other hand, since $\tilde{M}_{\tilde{g}}$ is bounded we have 
$$\sup_{\lambda\in\Lambda}|\tilde{g}(\lambda)|
=\sup_{\lambda\in\Lambda}\Vert\tilde{M}_{\tilde{g}}(e_\lambda)\Vert_{\ell_\infty(\Lambda)}
\leq\Vert\tilde{M}_{\tilde{g}}\Vert\Vert e_\lambda\Vert_{\ell_p(\Lambda)}
=\Vert\tilde{M}_{\tilde{g}}\Vert<+\infty
$$
Contradiction.

1.3) Consider subcase $p=+\infty$. Since $\Lambda$ is countable then $\ell_\infty(\Lambda)$ is non separable and $\ell_q(\Lambda)$ is separable. As $\tilde{M}_{\tilde{g}}$ is topologically injective, then $\operatorname{Im}(\tilde{M}_{\tilde{g}})$ is non separable subspace of $\ell_q(\Lambda)$. Contradiction.

In all subcases we got contradiction, hence $\Lambda$ is finite i.e.  $(\Omega,\Sigma,\mu)$ have finitely many atoms. From lemma \ref{FuncDescOnAtom} we see that $g$ is completely determined by its values $k_\lambda\in\mathbb{C}$ on atoms $\{\Omega_\lambda:\lambda\in\Lambda\}$. By proposition \ref{MultpOpSurjInjDesc} the function $g$ is zero only on sets of measre zero, so $k_\lambda\neq 0$ for all $\lambda\in\Lambda$. Since $\Lambda$ is finite we conclude $|g|\geq c$ for $c=\min_{\lambda\in\Lambda}|k_\lambda|$.

2) Let $p=q$. Fix $\lambda\in\Lambda$, then
$$
|\tilde{g}(\lambda)|
=\Vert \tilde{g}\cdot e_\lambda\Vert_{\ell_q(\Lambda)}
=\Vert \tilde{M}_{\tilde{g}}(e_\lambda)\Vert_{\ell_q(\Lambda)}
\geq c'\Vert e_\lambda\Vert_{\ell_p(\Lambda)}
=c'
$$
From lemma \ref{FuncDescOnAtom} for $\mu$-almost all $\omega\in\Omega_\lambda$ we have
$$
|g(\omega)|
=\left|\mu(\Omega_\lambda)^{-1}\int_{\Omega_\lambda}g(\omega)d\mu(\omega)\right|
=\left|\mu(\Omega_\lambda)^{-1}\mu(\Omega_\lambda)^{1+1/p-1/p}\tilde{g}(\lambda)\right|
=|\tilde{g}(\lambda)|\geq c'
$$
Since $\lambda\in\Lambda$ is arbitrary and $\Omega=\bigcup_{\lambda\in\Lambda}\Omega_\lambda$, then $|g|\geq c'$.

$ii)\implies i)$ Assume $|g|\geq c$ for $c>0$. Then from proposition \ref{EquivMultOp} we see that $|\tilde{g}|\geq c$.

1) Let $p\neq q$. Then we additionally assume that $(\Omega,\Sigma,\mu)$ have finitely many atoms. Now from proposition \ref{DescOfLpSpOnPureAtomMeasSp} we get that $L_p(\Omega,\mu)$ is finite dimensional. From assumption on $g$ we see it is have no zero values. Hence operator $M_g$ is topologically injective. 

2) Let $p=q$, then for all $x\in\ell_p(\Lambda)$ we have
$$
\Vert \tilde{M}_{\tilde{g}}\Vert_{\ell_p(\Lambda)}=\Vert g\cdot x\Vert_{\ell_p(\Lambda)}\geq c\Vert x\Vert_{\ell_p(\Lambda)}
$$
so $\tilde{M}_{\tilde{g}}$ is topologically injective and so does $M_g$.
\end{proof}

\begin{proposition}\label{TopInjMultOpCharacOnNonAtomMeasSp} Let $(\Omega,\Sigma,\mu)$ be a nonatomic measure space, $p,q\in[1,+\infty]$ and $g\in L_0(\Omega,\mu)$ then the following are equivalent

i) $M_g\in\mathcal{B}(L_p(\Omega,\mu),L_q(\Omega,\mu))$ is topologically injective 

ii) $|g|\geq c$ for some $c>0$ and $p=q$.
\end{proposition}
\begin{proof}
$i)\implies ii)$ Assume $M_g$ is topologically injective i.e. $\Vert M_g(f)\Vert_{L_q(\Omega,\mu)}\geq c\Vert f\Vert_{L_p(\Omega,\mu)}$ for some $c>0$ and all $f\in L_p(\Omega,\mu)$. We will consider three cases.

1)  Let $p>q$. There exist $C>0$ and $E\in\Sigma$ with $\mu(E)>0$ such that $|g|_E|\leq C$, otherwise $M_g$ is not well defined. By theorem \ref{ContOfNonAtmMeas} we have a sequence $\{E_n:n\in\mathbb{N}\}\subset\Sigma$ of subsets of $E$ such that $\mu(E_n)=2^{-n}$. Then since $p>q$ we get
$$
c
\leq\frac{\Vert M_g(\chi_{E_n})\Vert_{L_q(\Omega,\mu)}}{\Vert \chi_{E_n}\Vert_{L_p(\Omega,\mu)}}
\leq\frac{C\chi_{E_n}\Vert_{L_q(\Omega,\mu)}}{\Vert \chi_{E_n}\Vert_{L_p(\Omega,\mu)}}
\leq C\mu(E_n)^{1/q-1/p}
$$
$$
c
\leq\inf_{n\in\mathbb{N}}C\mu(E_n)^{1/q-1/p}
=C\inf_{n\in\mathbb{N}} 2^{n(1/p-1/q)}=0
$$
Contradiction, so in this case $M_g$ can not be topologically injective.

2) Let $p=q$. Fix $c'<c$. Assume there exist $E\in\Sigma$ with $\mu(E)>0$ and $|g|_{E}|<c'$, then
$$
\Vert M_g(\chi_{E})\Vert_{L_p(\Omega,\mu)}
=\Vert g \cdot\chi_{E}\Vert_{L_p(\Omega,\mu)}
\leq c' \Vert \chi_{E}\Vert_{L_p(\Omega,\mu)}
<c\Vert \chi_{E}\Vert_{L_p(\Omega,\mu)}
$$
Contradiction. Since $c'<c$ is arbitrary we conclude $|g|_E|\geq c$ for any $E\in\Sigma$ with $\mu(E)>0$. Thus $|g|\geq c$.

3) Let $p<q$. Assume we have some $c'>0$ and $E\in\Sigma$ such that $\mu(E)>0$, $|g|_E|>c'$. By theorem \ref{ContOfNonAtmMeas} we have a sequence  $\{E_n:n\in\mathbb{N}\}\subset\Sigma$ of subsets of $E$ such that $\mu(E_n)=2^{-n}$. Then from inequality $p<q$ we get
$$
\Vert M_g\Vert
\geq\frac{\Vert M_g(\chi_{E_n})\Vert_{L_q(\Omega,\mu)}}{\Vert \chi_{E_n}\Vert_{L_p(\Omega,\mu)}}
\geq\frac{c'\Vert\chi_{E_n}\Vert_{L_q(\Omega,\mu)}}{\Vert \chi_{E_n}\Vert_{L_p(\Omega,\mu)}}
\geq c'\mu(E_n)^{1/q-1/p}
$$
$$
\Vert M_g\Vert
\geq\sup_{n\in\mathbb{N}}c'\mu(E_n)^{1/q-1/p}
\geq c'\sup_{n\in\mathbb{N}}2^{n(1/p-1/q)}
=+\infty
$$
Contradiction, hence $g=0$. In this case by proposition \ref{MultpOpSurjInjDesc} operator $M_g$ is not topologically injective.

$ii)\implies i)$ Conversely, assume $|g|\geq c$ for $c>0$ and $p=q$. Then for all $f\in L_p(\Omega,\mu)$ we have
$$
\Vert M_g(f)\Vert_{L_p(\Omega,\mu)}
=\Vert g\cdot f\Vert_{L_p(\Omega,\mu)}
\geq c\Vert f\Vert_{L_p(\Omega,\mu)}
$$
So $M_g$ is topologically injective.
\end{proof}

\begin{theorem}\label{TopInjMultOpCharacOnMeasSp} Let $(\Omega,\Sigma,\mu)$ be a $\sigma$-finite measure space, $p,q\in[1,+\infty]$ and $g\in L_0(\Omega,\Sigma,\mu)$, then the following are equivalent 

i) $M_g\in\mathcal{B}(L_p(\Omega,\mu),L_q(\Omega,\mu))$ is toplogically injective 

ii) $M_g$ is an isomorphism 

iii) $|g|\geq c$ for some $c>0$, if $p\neq q$ the space $(\Omega,\Sigma,\mu)$ consist of finitely many atoms
\end{theorem}
\begin{proof} $i)\Longleftrightarrow iii)$ By proposition \ref{DescOfLpSpOnMeasSp} we have decomposition $\Omega=\Omega_{a}\cup\Omega_{na}$, where $(\Omega_{na},\Sigma|_{\Omega_{na}},\mu|_{\Omega_{na}})$ is a nonatomic measure space and $(\Omega_{a},\Sigma|_{\Omega_{a}},\mu|_{\Omega_{a}})$ is a purely atomic measure space. By proposition \ref{MultOpDecompDecomp} operator $M_g$ is topologically injective if and only if so does $M_g^{\Omega_{a}}$ and $M_g^{\Omega_{na}}$. Propositions \ref{TopInjMultOpCharacOnPureAtomMeasSp}, \ref{TopInjMultOpCharacOnNonAtomMeasSp} give necessary and sufficient conditions for $M_g^{\Omega_{a}}$ and $M_g^{\Omega_{na}}$ to be topologically injective. So we get the result. 

$i)\implies ii)$ Assume $M_g$ is topologically injective. If $p=q$ from considerations above it follows that $|g|\geq c$ for some $c>0$. Since $M_g$ is bounded from proposition \ref{MulpOpPropIfPeqqualsQ} we also have $C\geq |g|$ for some $C>0$. Now from the same proposition we conclude that $M_g$ is an isomorphism because $C\geq|g|\geq c$. Assume $p\neq q$, then from previous paragraph the space $(\Omega,\Sigma,\mu)$ consist of finite amount of atoms and $g$ is injective. Hence from proposition \ref{DescOfLpSpOnPureAtomMeasSp} we get $\operatorname{dim}(L_p(\Omega,\Sigma,\mu))=\operatorname{dim}(\ell_p(\Lambda))=\operatorname{Card}(\Lambda)<+\infty$. Similarly, $\operatorname{dim}(L_q(\Omega,\Sigma,\mu))=\operatorname{Card}(\Lambda)<+\infty$ Since $g$ is injective by proposition \ref{MultpOpSurjInjDesc} operator $M_g$ is injective. Thus $M_g$ is an injective operator between finite dimensional spaces of equal dimension. Hence it is an isomorphism.

$ii)\implies i)$ Conversely, if $M_g$ is an isomorphism, clearly, it is topologically injective.
\end{proof}

\begin{proposition}\label{TopInjMultOpCharacBtwnTwoContMeasSp} Let $(\Omega,\Sigma,\mu)$ be a $\sigma$-finite measure space, $p,q\in[1,+\infty]$ and $g,\rho\in L_0(\Omega,\mu)$ and $\rho$ is non negative then the following are equivalent 

i) $M_g\in\mathcal{B}(L_p(\Omega,\mu),L_q(\Omega,\rho\cdot\mu))$ is topologically injective 

ii) $M_g$ is an isomorphism 

iii) $\rho$ is  positive, $|g\cdot \rho^{1/q}|\geq c$ for some $c>0$, if $p\neq q$ the space $(\Omega,\Sigma,\mu)$ consist of finitely many atoms.
\end{proposition}
\begin{proof} $i)\implies iii)$ Consider set $E=\rho^{-1}(\{0\})$. Assume $\mu(E)>0$ then $\chi_E\neq 0$ in $L_p(\Omega,\mu)$. On the other hand $(\rho\cdot\mu)(E)=\int_E\rho(\omega)d\mu(\omega)=0$, so $\chi_E=0$ in $L_q(\Omega,\rho\cdot\mu)$ and $M_g(\chi_E)=g\cdot\chi_E=0$ in $L_q(\Omega,\rho\cdot\mu)$. Thus we see that $M_g$ is not injective and as the consequence it is not topologically injective. Contradiction, so $\mu(E)=0$ and $\rho$ is  positive. Hence by proposition \ref{ChngOfDenst} we have an isometric isomorphism $\bar{I}_q:L_q(\Omega,\mu)\to L_q(\Omega,\rho\cdot\mu),f\mapsto \rho^{-1/q}\cdot f$. Obviously $M_{g\cdot\rho^{1/q}}=\bar{I}_q^{-1} M_g\in\mathcal{B}(L_p(\Omega,\mu),L_q(\Omega,\mu))$. Since $\bar{I}_q$ is an isometric isomorphism and $M_g$ is topologically injective, then $M_{g\cdot \rho^{1/q}}$ is topologically injective. From theorem \ref{TopInjMultOpCharacOnMeasSp} we get that $|g\cdot\rho^{1/q}|\geq c$ for some $c>0$ and if $p\neq q$ the space is $(\Omega,\Sigma,\mu)$ consist of finite amount of atoms.

$iii)\implies i)$ By theorem \ref{TopInjMultOpCharacOnMeasSp} operator $M_{g\cdot\rho^{1/q}}$ is topologically injective. Since $\rho$ is positive by proposition \ref{ChngOfDenst} we have an isometric isomorohism $\bar{I}_q$. Then from equality $M_g=\bar{I}_q M_{g\cdot\rho^{1/q}}$ it follows that $M_g$ is also topologically injective.

$i)\implies ii)$ As we proved above this implies that $M_{g\cdot\rho^{1/q}}$ is topologically injective and $\bar{I}_q$ is an isometric isomorphism. By theorem \ref{TopInjMultOpCharacOnMeasSp} $M_{g\cdot\rho^{1/q}}$ is an isomorphism. Since $M_g=\bar{I}_q M_{g\cdot\rho^{1/q}}$ and $\bar{I}_q$ is an isometric isomorphism, then $M_g$ is also an isomorphism.

$ii)\implies i)$ If $M_g$ is an isomorphism, then, obviously, it is topologically injective.
\end{proof}

\begin{theorem}\label{TopInjMultOpCharacBtwnTwoMeasSp} Let $(\Omega,\Sigma,\mu)$, $(\Omega,\Sigma,\nu)$ be two $\sigma$-finite measure spaces, $p,q\in[1,+\infty]$ and $g\in L_0(\Omega,\mu)$, then the following are eqivalent

i) $M_g\in\mathcal{B}(L_p(\Omega,\mu), L_q(\Omega,\nu))$ is topologically injective 

ii) $M_g^{\Omega_c}$ is an isomorphism

iii) $\rho_{\nu,\mu}$ is positive, $|g\cdot\rho_{\nu,\mu}^{1/q}|_{\Omega_c}|\geq c$ for some $c>0$, if $p\neq q$ the space $(\Omega,\Sigma,\mu)$ consist of finitely many atoms.
\end{theorem}
\begin{proof}
By proposition \ref{MultOpDecompDecomp} operator $M_g$ is topologically injective if and only if operators $M_g^{\Omega_c}:L_p(\Omega_c,\mu|_{\Omega_c})\to L_q(\Omega_c,\rho_{\nu,\mu}\cdot\mu|_{\Omega_c})$ and $M_g^{\Omega_s}:L_p(\Omega_s,\mu|_{\Omega_s})\to L_q(\Omega_s,\nu_s|_{\Omega_s})$ are topologically injective. By proposition \ref{MultOpCharacBtwnTwoSingMeasSp} operator $M_g^{\Omega_s}$ is zero. Since $\mu(\Omega_s)=0$, then $L_p(\Omega_s,\mu|_{\Omega_s})=\{0\}$. From these two facts we conclude that $M_g^{\Omega_s}$ is topologically injective. Thus topological injectivity of $M_g$ is equivalent to topological injectivity of  $M_g^{\Omega_c}$. It is remains to apply proposition \ref{TopInjMultOpCharacBtwnTwoContMeasSp}.
\end{proof}

\begin{theorem}\label{TopInjMultOpDescBtwnTwoMeasSp} Let $(\Omega,\Sigma,\mu)$, $(\Omega,\Sigma,\nu)$ be two $\sigma$-finite measure spaces, $p,q\in[1,+\infty]$ and $g\in L_0(\Omega,\mu)$, then the following are equivalent 

i) $M_g\in\mathcal{B}(L_p(\Omega,\mu),L_q(\Omega,\nu))$ is topologically injective 

ii) $M_{\chi_{\Omega_c}/g}\in\mathcal{B}(L_q(\Omega,\nu), L_p(\Omega,\mu))$ its left inverse topologically surjective operator
\end{theorem}
\begin{proof}
$i)\implies ii)$ By proposition \ref{MultOpDecompDecomp}  $M_g^{\Omega_c}$ is topologically injective. By proposition \ref{TopInjMultOpCharacBtwnTwoContMeasSp} operator $M_g^{\Omega_c}$ is invertible and $(M_g^{\Omega_c})^{-1}=M_{1/g}^{\Omega_c}$. Then for all $h\in L_q(\Omega,\nu)$ we have
$$
\Vert M_{\chi_{\Omega_c}/g}(h)\Vert_{L_p(\Omega,\mu)}=
\Vert M_{1/g}(h)\chi_{\Omega_c}\Vert_{L_p(\Omega,\mu)}=
\Vert M_{1/g}^{\Omega_c}(h|_{\Omega_c})\Vert_{L_p(\Omega_c,\mu|_{\Omega_c})}
$$
$$
\leq\Vert M_{1/g}^{\Omega_c}\Vert\Vert h|_{\Omega_c}\Vert_{L_q(\Omega_c,\nu|_{\Omega_c})}
\leq\Vert M_{1/g}^{\Omega_c}\Vert\Vert h\Vert_{L_q(\Omega,\nu)}
$$ 
So $M_{\chi_{\Omega_c}/g}$ is bounded. Now note that for all $f\in L_p(\Omega,\mu)$ we have 
$$
M_{\chi_{\Omega_c}/g}(M_g(f))
=M_{\chi_{\Omega_c}/g}(g\cdot f)
=(\chi_{\Omega_c}/g)\cdot g\cdot f
=f\cdot\chi_{\Omega_c}
$$
Since $\mu(\Omega\setminus\Omega_c)=0$, then $\chi_{\Omega_c}=\chi_{\Omega}$, so $M_{\chi_{\Omega_c}/g}(M_g(f))=f\cdot\chi_{\Omega_c}=f\cdot\chi_{\Omega}=f$. This means that $M_g$ have left inverse multiplication operator. Take any $f\in L_p(\Omega,\mu)$, then for $h=M_g(f)$ we have $M_{\chi_{\Omega_c}/g}(h)=f$ and $\Vert h\Vert_{L_q(\Omega,\nu)}\leq\Vert M_g\Vert\Vert f\Vert_{L_p(\Omega,\mu)}$. Since $h$ is arbitrary $M_{\chi_{\Omega_c}/g}$ is topologically surjective.

Conversely if $M_g$ have left inverse $M_{\chi_{\Omega_c}/g}$ then for all $f\in L_p(\Omega,\mu)$ we have 
$$
\Vert M_g(f)\Vert_{L_q(\Omega,\nu)}
\geq\Vert M_{\chi_{\Omega_c}/g}\Vert^{-1}\Vert M_{\chi_{\Omega_c}/g}(M_g(f))\vert_{L_p(\Omega,\mu)}
\geq\Vert M_{\chi_{\Omega_c}/g}\Vert^{-1}\Vert f\Vert_{L_p(\Omega,\mu)}
$$
So $M_g$ is topologically injective.
\end{proof}

\subsection{Isometric operators}

Obviously isometric operators are particular case of topologically injective ones, so we can use certain results of previous section.  

\begin{proposition}\label{IsomMultOpCharacOnMeasSp} Let $(\Omega,\Sigma,\mu)$ be a $\sigma$-finite measure space, $p,q\in[1,+\infty]$ and $g\in L_0(\Omega,\mu)$, then the following are equivalent 

i) $M_g\in\mathcal{B}(L_p(\Omega,\mu),L_q(\Omega,\mu))$ is an isometry 

ii) $|g|=\mu(\Omega)^{1/p-1/q}$, if $p\neq q$, then $(\Omega,\Sigma,\mu)$ consist of single atom.

\end{proposition}
\begin{proof} $i)\implies ii)$ Let $p=q$. Assume there exist $E\in\Sigma$ with $\mu(E)>0$ such that $|g|_E|<1$, then
$$
\Vert M_g(\chi_E)\Vert_{L_p(\Omega,\mu)}
=\Vert g\cdot\chi_E\Vert_{L_p(\Omega,\mu)}
<\Vert\chi_E\Vert_{L_p(\Omega,\mu)}
=\Vert M_g(\chi_E)\Vert_{L_p(\Omega,\mu)}
$$
Contradiction, hence for all $E\in\Sigma$ with $\mu(E)>0$ we have $|g|_E|\geq 1$ i.e.  $|g|\geq 1$. Assume there exist $E\in\Sigma$ with $\mu(E)>0$ such that $|g|_E|>1$, then
$$
\Vert M_g(\chi_E)\Vert_{L_p(\Omega,\mu)}
=\Vert g\cdot\chi_E\Vert_{L_p(\Omega,\mu)}
>\Vert\chi_E\Vert_{L_p(\Omega,\mu)}
=\Vert M_g(\chi_E)\Vert_{L_p(\Omega,\mu)}
$$
Contradiction, hence for all $E\in\Sigma$ with $\mu(E)>0$ we have $|g|_E|\leq 1$ i.e.  $|g|\leq 1$. From both inequalities we get $|g|=1=\mu(\Omega)^{1/p-1/q}$. Let $p\neq q$, then since $M_g$ is an isometry it is topologically injective. By theorem \ref{TopInjMultOpCharacOnMeasSp} the space $(\Omega,\Sigma,\mu)$ consist of finitely many atoms. Assume there is at least two disjoint atoms, say $\Omega_1$ and $\Omega_2$. By lemma \ref{AtomDescInSigmFinMeasSp} they are of finite measure, so we can consider respective normalized functions $h_k=\Vert\chi_{\Omega_k}\Vert_{L_p(\Omega,\mu)^{-1}}\chi_{\Omega_k}$ where $k\in\{1,2\}$. Since they these atoms are disjoint $h_1h_2=0$ and as the result $M_g(h_1)M_g(h_2)=0$. Note that for any $r\in[1,+\infty]$ and all $f_1,f_2\in L_r(\Omega,\mu)$ such that $f_1f_2=0$ we have 
$$
\Vert f_1+f_2\Vert_{L_r(\Omega,\mu)}
=\left\Vert\left(\Vert f_\lambda\Vert_{L_r(\Omega,\mu)}:\lambda\in\{1,2\}\right)\right\Vert_{\ell_r(\{1,2\})}
$$
Hence
$$
\Vert M_g(h_1+h_2)\Vert_{L_q(\Omega,\mu)}
=\Vert h_1+h_2\Vert_{L_p(\Omega,\mu)}
=\left\Vert\left( 1 :\lambda\in\{1,2\}\right)\right\Vert_{\ell_p(\{1,2\})}
=2^{1/p}
$$
But on the other hand
$$
\Vert M_g(h_1+h_2)\Vert_{L_q(\Omega,\mu)}
=\Vert M_g(h_1)+M_g(h_2)\Vert_{L_q(\Omega,\mu)}
=\left\Vert\left(\Vert M_g(h_\lambda)\Vert_{L_q(\Omega,\mu)}:\lambda\in\{1,2\}\right)\right\Vert_{\ell_q(\{1,2\})}
$$
$$
=\left\Vert\left(\Vert h_\lambda\Vert_{L_p(\Omega,\mu)}:\lambda\in\{1,2\}\right)\right\Vert_{\ell_q(\{1,2\})}
=\left\Vert\left(1:\lambda\in\{1,2\}\right)\right\Vert_{\ell_q(\{1,2\})}
=2^{1/q}
$$
Thus $2^{1/p}=2^{1/q}$. Contradiction, so $(\Omega,\Sigma,\mu)$ consist of single atom. In this case from lemma \ref{FuncDescOnAtom} it follows that for all $f\in L_p(\Omega,\mu)$ we have
$$
\Vert M_g(f)\Vert_{L_q(\Omega,\mu)}
=\Vert J_q(M_g(f))\Vert_{\ell_q(\{1\})}
=\Vert J_q(g\cdot f)\Vert_{\ell_q(\{1\})}
=\mu(\Omega)^{1/q-1}\left|\int_\Omega g(\omega) f(\omega)d\mu(\omega)\right|
$$
$$
\Vert f\Vert_{L_p(\Omega,\mu)}
=\Vert J_p(f)\Vert_{\ell_p(\{1\})}
=\mu(\Omega)^{1/p-1}\left|\int_\Omega f(\omega)d\mu(\omega)\right|
$$
By $c$ we denote the constant value of $g$, then
$$
\Vert M_g(f)\Vert_{L_q(\Omega,\mu)}
=\mu(\Omega)^{1/q-1}\left|\int_\Omega g(\omega) f(\omega)d\mu(\omega)\right|
=\mu(\Omega)^{1/q-1}|c|\left|\int_\Omega f(\omega)d\mu(\omega)\right|
$$
From this equality we conclude that in this case $M_g$ is an isometry if
$$
|g|=|c|=\mu(\Omega)^{1/p-1/q}
$$ 

$ii)\implies i)$. Let $p=q$, then $|g|=1$. So for all $f\in L_p(\Omega,\mu)$ we have
$$
\Vert M_g(f)\Vert_{L_p(\Omega,\mu)}
=\Vert g\cdot f\Vert_{L_p(\Omega,\mu)}
=\Vert |g|\cdot f\Vert_{L_p(\Omega,\mu)}
=\Vert f\Vert_{L_p(\Omega,\mu)}
$$
hence $M_g$ is an isometry. Let $p\neq q$, then $(\Omega,\Sigma,\mu)$ consist of single atom and we conclude
$$
\Vert M_g(f)\Vert_{L_q(\Omega,\mu)}
=\mu(\Omega)^{1/q-1}\left|\int_\Omega g(\omega) f(\omega)d\mu(\omega)\right|
=\mu(\Omega)^{1/q-1}|c|\left|\int_\Omega f(\omega)d\mu(\omega)\right|
$$
$$
=\mu(\Omega)^{1/p-1}\left|\int_\Omega f(\omega)d\mu(\omega)\right|
=\Vert f\Vert_{L_p(\Omega,\mu)}
$$
hence $M_g$ is isometric.
\end{proof}

\begin{proposition}\label{IsomMultOpCharacBtwnTwoContMeasSp} Let $(\Omega,\Sigma,\mu)$ be a $\sigma$-finite measure space, $p,q\in[1,+\infty]$ and $g,\rho\in L_0(\Omega,\mu)$, and $\rho$ is non negative, then the following are equivalent 

i) $M_g\in\mathcal{B}(L_p(\Omega,\mu), L_q(\Omega,\rho\cdot\mu))$ is isometric

ii) $M_g$ is isometric isomorphism

iii) $\rho$ is positive, $|g\cdot \rho^{1/q}|=\mu(\Omega)^{1/p-1/q}$ and if $p\neq q$ the sapce $(\Omega,\Sigma,\mu)$ consist of single atom.
\end{proposition}
\begin{proof} $i)\implies iii)$ Since $M_g$ is isometric, it is topologically injective and by theorem \ref{TopInjMultOpCharacBtwnTwoMeasSp} we see that $\rho$ is positive. Hence by proposition \ref{ChngOfDenst} we have an isometric isomorphism $\bar{I}_q:L_q(\Omega,\mu)\to L_q(\Omega,\rho\cdot\mu),f\mapsto \rho^{-1/q}\cdot f$. Obviously $M_{g\cdot\rho^{1/q}}=\bar{I}_q^{-1} M_g\in\mathcal{B}(L_p(\Omega,\mu),L_q(\Omega,\mu))$. Since $\bar{I}_q$ is an isometric isomorphism and $M_g$ is isometric, then $M_{g\cdot \rho^{1/q}}$ is isometric too. It is remains to apply theorem \ref{IsomMultOpCharacOnMeasSp}.

$iii)\implies i)$ By theorem \ref{IsomMultOpCharacOnMeasSp} operator $M_{g\cdot\rho^{1/q}}$ is isometric. Since $\rho$ is positive by proposition \ref{ChngOfDenst} we have an isometric isomorohism $\bar{I}_q$. Then from equality $M_g=\bar{I}_q M_{g\cdot\rho^{1/q}}$ it follows that $M_g$ is also isometric.

$i)\implies ii)$ Since $M_g$ is isometric, it is topologically injective and by proposition \ref{TopInjMultOpCharacBtwnTwoContMeasSp} it is an isomorphism, which is isometric by assumption.

$ii)\implies i)$ Since $M_g$ is an isometric isomorphism, trivially, it is isometric.
\end{proof}

\begin{theorem}\label{IsomMultOpCharacBtwnTwoMeasSp} Let $(\Omega,\Sigma,\mu)$, $(\Omega,\Sigma,\nu)$ be two $\sigma$-finite measure spaces, $p,q\in[1,+\infty]$ and $g\in L_0(\Omega,\mu)$, then the following are equivalent 

i) $M_g$ is isometric 

ii) $M_g^{\Omega_c}$ is isometric 

iii) $\rho_{\nu,\mu}$ is positive, $|g\cdot \rho_{\nu,\mu}^{1/q}|_{\Omega_c}|=\mu(\Omega_c)^{1/p-1/q}$ and if  $p\neq q$ the space $(\Omega,\Sigma,\mu)$ consist of single atom.
\end{theorem}
\begin{proof}
$i)\implies ii)\implies iii)$ Since $M_g$ is isometric, by proposition \ref{MultOpDecompDecomp} operator $M_g^{\Omega_c}$, is isometric. It is remains to apply proposition \ref{IsomMultOpCharacBtwnTwoContMeasSp}.

$iii)\implies i)$ By proposition \ref{IsomMultOpCharacBtwnTwoContMeasSp} operator $M_g^{\Omega_c}$ is isometric. Now take arbitrary $f\in L_p(\Omega,\mu)$. Since $\mu(\Omega\setminus\Omega_c)=0$, then $\chi_{\Omega_c}=\chi_{\Omega}$ in $L_p(\Omega,\mu)$. As the result $f=f\chi_{\Omega}=f\chi_{\Omega_c}=f\chi_{\Omega_c}\chi_{\Omega_c}$ in $L_p(\Omega,\mu)$ and $M_g(f)=M_g(f\chi_{\Omega_c})\chi_{\Omega_c}$. Thus using that $M_g^{\Omega_c}$ is isometric we get
$$
\Vert M_g(f)\Vert_{L_q(\Omega,\nu)}
=\Vert M_g(f\chi_{\Omega_c})\chi_{\Omega_c}\Vert_{L_q(\Omega,\nu)}
=\Vert M_g(f\chi_{\Omega_c})\Vert_{L_q(\Omega_c,\nu|_{\Omega_c})}
$$
$$
=\Vert M_g^{\Omega_c}(f|_{\Omega_c})\Vert_{L_q(\Omega_c,\nu|_{\Omega_c})}
=\Vert f|_{\Omega_c}\Vert_{L_p(\Omega_c,\mu|_{\Omega_c})}
$$
Since $\mu(\Omega\setminus\Omega_c)=0$ we have $\Vert f|_{\Omega_c}\Vert_{L_p(\Omega_c,\mu|_{\Omega_c})}=\Vert f\Vert_{L_p(\Omega,\mu)}$ so $\Vert M_g(f)\Vert_{L_q(\Omega,\nu)}=\Vert f\Vert_{L_p(\Omega,\mu)}$ i.e.  $M_g$ is isometric.
\end{proof}

\begin{theorem}\label{IsomMultOpDescBtwnTwoMeasSp} Let $(\Omega,\Sigma,\mu)$, $(\Omega,\Sigma,\nu)$ be two $\sigma$-finite measure spaces, $p,q\in[1,+\infty]$ and $g\in L_0(\Omega,\mu)$, then the following are equivalent 

i) $M_g\in\mathcal{B}(L_p(\Omega,\mu),L_q(\Omega,\nu))$ is isometric 

ii) $M_{\chi_{\Omega_c}/g}\in\mathcal{B}(L_q(\Omega,\nu), L_p(\Omega,\mu))$ its left inverse strictly coisometric operator.
\end{theorem}
\begin{proof}
$i)\implies ii)$ By proposition \ref{MultOpDecompDecomp} operator $M_g^{\Omega_c}$ is isometric and by proposition \ref{IsomMultOpCharacBtwnTwoContMeasSp} it is invertible with $(M_g^{\Omega_c})^{-1}=M_{1/g}^{\Omega_c}$. Since $M_g^{\Omega_c}$ is isometric then so does its inverse. Then for all $h\in L_q(\Omega,\nu)$ we have
$$
\Vert M_{\chi_{\Omega_c}/g}(h)\Vert_{L_p(\Omega,\mu)}=
\Vert M_{1/g}(h|_{\Omega_c})\Vert_{L_p(\Omega_c,\mu|_{\Omega_c})}=
\Vert M_{1/g}^{\Omega_c}(h|_{\Omega_c})\Vert_{L_p(\Omega_c,\mu|_{\Omega_c})}
$$
$$
=\Vert h|_{\Omega_c}\Vert_{L_q(\Omega_c,\nu|_{\Omega_c})}
\leq \Vert h \Vert_{L_q(\Omega,\nu)}
$$
So $M_{\chi_{\Omega_c}/g}$ is contractive. Now note that for all $f\in L_p(\Omega,\mu)$ we have 
$$
M_{\chi_{\Omega_c}/g}(M_g(f))
=M_{\chi_{\Omega_c}/g}(g\cdot f)
=(\chi_{\Omega_c}/g)\cdot g\cdot f
=f\cdot\chi_{\Omega_c}
$$
Since $\mu(\Omega\setminus\Omega_c)=0$, then $\chi_{\Omega_c}=\chi_{\Omega}$, so $M_{\chi_{\Omega_c}/g}(M_g(f))=f\cdot\chi_{\Omega_c}=f\cdot\chi_{\Omega}=f$. This means that $M_g$ have left inverse multiplication operator. Take any $f\in L_p(\Omega,\mu)$, then for $h=M_g(f)$ we have $M_{\chi_{\Omega_c}/g}(h)=f$ and $\Vert h\Vert_{L_q(\Omega,\nu)}\leq\Vert f\Vert_{L_p(\Omega,\mu)}$ i.e.  $M_{\chi_{\Omega_c}}/g$ is strictly $1$-topologigically surjective. Since $M_{\chi_{\Omega_c}/g}$ is also contractive, it is strictly coisometric.

$ii)\implies i)$ Take any $f\in L_p(\Omega,\mu)$, then there exist $h\in L_q(\Omega,\nu)$ such that $M_{\chi_{\Omega_c}/g}(h)=f$ and $\Vert h\Vert_{L_q(\Omega,\nu)}\leq \Vert f\Vert_{L_p(\Omega,\mu)}$. Hence
$$
\Vert M_g(f)\Vert_{L_q(\Omega,\nu)}
=\Vert M_g(M_{\chi_{\Omega_c}/g}(h))\Vert_{L_q(\Omega,\nu)}
=\Vert \chi_{\Omega_c}h\Vert_{L_q(\Omega,\nu)}
\leq\Vert h\Vert_{L_q(\Omega,\nu|)}
\leq\Vert f\Vert_{L_p(\Omega,\mu)}
$$
Since $M_{\chi_{\Omega_c}/g}$ is contractive and left inverse to $M_g$ then
$$
\Vert f\Vert_{L_p(\Omega,\mu)}
=\Vert M_{\chi_{\Omega_c}/g}(M_g(f))\Vert_{L_p(\Omega,\mu)}
\leq\Vert M_g(f)\Vert_{L_q(\Omega,\nu)}
$$
so $\Vert M_g(f)\Vert_{L_q(\Omega,\nu)}=\Vert f\Vert_{L_p(\Omega,\mu)}$. Since $f$ is arbitrary $M_g$ is isometric.
\end{proof}

\subsection{Topologically surjective operators}

Description of topologically surjective operators is slightly easier to obtain. We will show that all such operators are isomorphisms or invertible from the right. Most of the proofs goes along the lines of previous sections.

\begin{theorem}\label{TopSurMultOpCharacOnMeasSp} Let $(\Omega,\Sigma,\mu)$ be a $\sigma$-finite measure space, $p,q\in[1,+\infty]$ and $g\in L_0(\Omega,\mu)$, then the following are equivalent

i) $M_g\in\mathcal{B}(L_p(\Omega,\mu),L_q(\Omega,\mu))$ is topologically surjective 

ii) $M_g$ is an isomorphism

iii) $|g|\geq c$ for some $c>0$, if $p\neq q$ the space $(\Omega,\Sigma,\mu)$ consist of finitely many atoms.
\end{theorem}
\begin{proof} $i)\implies iii)$ Since $M_g$ be topologically surjective, then it is surjective and by proposition \ref{MultpOpSurjInjDesc} it is also injective. Thus $M_g$ is bijective. Since $L_p$ spaces are complete, from open mapping theorem we see that $M_g$ is an isomorphism. 

$ii)\implies i)$ If $M_g$ is an isomorphism, obviously, it is topologically surjective.

$i)\Longleftrightarrow iii)$ Follows from theorem \ref{TopInjMultOpCharacOnMeasSp}
\end{proof}
 
\begin{proposition}\label{TopSurMultOpCharacBtwnTwoContMeasSp} Let $(\Omega,\Sigma,\nu)$ be a $\sigma$-finite measure space, $p,q\in[1,+\infty]$ and $g,\rho\in L_0(\Omega,\rho\cdot\nu)$ and $\rho$ is non negative, then the following are equivalent

i) $M_g\in\mathcal{B}(L_p(\Omega,\rho\cdot\nu),L_q(\Omega,\nu))$ is topologically surjective 

2) $M_g$ is an isomorphism 

3) $\rho$ is positive, $|g\cdot \rho^{-1/p}|\geq c$ for some $c>0$, if $p\neq q$ the space $(\Omega,\Sigma,\mu)$ consist of finitely many atoms.
\end{proposition}
\begin{proof} $i)\implies iii)$ Consider set $E=\rho^{-1}(\{0\})$. Assume $\nu(E)>0$ then $\chi_E\neq 0$ in $L_p(\Omega,\nu)$. On the other hand $(\rho\cdot\nu)(E)=\int_E\rho(\omega)d\nu(\omega)=0$, so $\chi_E=0$ in $L_q(\Omega,\rho\cdot\mu)$. Then for all $f\in L_p(\Omega,\rho\cdot\nu)$ holds $M_g(f)\chi_E=M_g(f\cdot\chi_E)=M_g(0)=0$ in $L_q(\Omega,\nu)$. The last equality means that $\operatorname{Im}(M_g)\subset\{h\in L_q(\Omega,\mu): h|_E=0\}$. Since $\nu(E)\neq 0$ we see that $M_g$ is not surjective and as the consequence it is not topologically surjective. Contradiction, so $\nu(E)=0$ and $\rho$ is positive. Hence by proposition \ref{ChngOfDenst} we have an isometric isomorphism $\bar{I}_p:L_p(\Omega,\nu)\to L_p(\Omega,\rho\cdot\nu),f\mapsto \rho^{-1/p}\cdot f$. Obviously $M_{g\cdot\rho^{-1/p}}=M_g \bar{I}_p\in\mathcal{B}(L_p(\Omega,\nu),L_q(\Omega,\nu))$. Since $\bar{I}_p$ is an isometric isomorphism and $M_g$ is topologically surjective, then $M_{g\cdot \rho^{-1/p}}$ is topologically surjective. It is remains to apply theorem \ref{TopSurMultOpCharacOnMeasSp}.

$iii)\implies i)$ By theorem \ref{TopSurMultOpCharacOnMeasSp} operator $M_{g\cdot\rho^{-1/p}}$ is topologically surjective. Since $\rho$ is positive by proposition \ref{ChngOfDenst} we have an isometric isomorohism $\bar{I}_p$. Then from equality $M_g= M_{g\cdot\rho^{-1/p}}\bar{I}_p^{-1}$ it follows that $M_g$ is also topologically surjective.

$i)\implies ii)$ As we proved above this operator $M_{g\cdot\rho^{1/q}}$ is topologically injective and $\bar{I}_q$ is an isometric isomorphism. By theorem \ref{TopSurMultOpCharacOnMeasSp} $M_{g\cdot\rho^{1/q}}$ is an isomorphism. Since $M_g=\bar{I}_q M_{g\cdot\rho^{1/q}}$ we see that $M_g$ is also an isomorphism, as composition of such.

$ii)\implies i)$. If $M_g$ is an isomorphism, obviously, it is topologically surjective.
\end{proof}

\begin{theorem}\label{TopSurMultOpCharacBtwnTwoMeasSp} Let $(\Omega,\Sigma,\mu)$, $(\Omega,\Sigma,\nu)$ be two $\sigma$-finite measure spaces, $p,q\in[1,+\infty]$ and $g\in L_0(\Omega,\mu)$, then the following are equivalent

i) $M_g\in\mathcal{B}(L_p(\Omega,\mu), L_q(\Omega,\nu))$ is topologically surjective 

ii) $M_g^{\Omega_c}$ is topologically surjective 

iii) $\rho_{\mu,\nu}$ is positive, $|g\cdot\rho_{\mu,\nu}^{-1/p}|_{\Omega_c}|\geq c$ for some $c>0$, if $p\neq q$ the space $(\Omega,\Sigma,\mu)$ consist of finitely many atoms.
\end{theorem}
\begin{proof}
By proposition \ref{MultOpDecompDecomp} operator $M_g$ is topologically surjective if and only if operators $M_g^{\Omega_c}:L_p(\Omega_c,\rho_{\mu,\nu}\cdot\nu|_{\Omega_c})\to L_q(\Omega_c,\nu|_{\Omega_c})$ and $M_g^{\Omega_s}:L_p(\Omega_s,\mu_s|_{\Omega_s})\to L_q(\Omega_s,\nu|_{\Omega_s})$ are topologically surjective. By proposition \ref{MultOpCharacBtwnTwoSingMeasSp} operator $M_g^{\Omega_s}$ is zero. Since $\nu(\Omega_s)=0$, then $L_p(\Omega_s,\nu|_{\Omega_s})=\{0\}$. From these two facts we conclude that $M_g^{\Omega_s}$ is topologically surjective. Thus topological surjectivity of $M_g$ is equivalent to topological injectivity of  $M_g^{\Omega_c}$. It is remains to apply proposition \ref{TopSurMultOpCharacBtwnTwoContMeasSp}.
\end{proof}

\begin{theorem}\label{TopSurMultOpDescBtwnTwoMeasSp} Let $(\Omega,\Sigma,\mu)$, $(\Omega,\Sigma,\nu)$ be two $\sigma$-finite measure spaces, $p,q\in[1,+\infty]$ and $g\in L_0(\Omega,\mu)$, then the following are equivalent 

i) $M_g\in\mathcal{B}(L_p(\Omega,\mu),L_q(\Omega,\nu))$ is topologically surjective 

ii) $M_{\chi_{\Omega_c}/g}\in\mathcal{B}(L_q(\Omega,\nu), L_p(\Omega,\mu))$ its right inverse topologically injective operator.
\end{theorem}
\begin{proof}
$i)\implies ii)$ By proposition \ref{MultOpDecompDecomp} operator $M_g^{\Omega_c}$ is topologically surjective. By proposition \ref{TopSurMultOpCharacBtwnTwoContMeasSp} it is invertable and $(M_g^{\Omega_c})^{-1}=M_{1/g}^{\Omega_c}$. Then for all $h\in L_q(\Omega,\nu)$ we have
$$
\Vert M_{\chi_{\Omega_c}/g}(h)\Vert_{L_p(\Omega,\mu)}=
\Vert M_{1/g}(h|_{\Omega_c})\Vert_{L_p(\Omega_c,\mu|_{\Omega_c})}=
\Vert M_{1/g}^{\Omega_c}(h|_{\Omega_c})\Vert_{L_p(\Omega_c,\mu|_{\Omega_c})}
$$
$$
\leq\Vert M_{1/g}^{\Omega_c}\Vert\Vert h|_{\Omega_c}\Vert_{L_q(\Omega_c,\nu|_{\Omega_c})}
\leq\Vert M_{1/g}^{\Omega_c}\Vert\Vert h\Vert_{L_q(\Omega,\nu)}
$$ 
So $M_{\chi_{\Omega_c}/g}$ is bounded. Now note that for all $h\in L_q(\Omega,\nu)$ we have 
$$
M_g(M_{\chi_{\Omega_c}/g}(h))
=M_g(\chi_{\Omega_c}/g\cdot h)
=g\cdot(\chi_{\Omega_c}/g)\cdot  h
=h\cdot\chi_{\Omega_c}
$$
Since $\nu(\Omega\setminus\Omega_c)=0$, then $\chi_{\Omega_c}=\chi_{\Omega}$, so $M_g(M_{\chi_{\Omega_c}/g}(h))=h\cdot\chi_{\Omega_c}=h\cdot\chi_{\Omega}=h$. This means that $M_g$ have right inverse multiplication operator. Take any $h\in L_q(\Omega,\nu)$, then
$$
\Vert M_{\chi_{\Omega_c}/g}(h)\Vert_{L_p(\Omega,\mu)}
\geq\Vert M_g\Vert\Vert M_g(M_{\chi_{\Omega_c}/g}(h))\Vert_{L_q(\Omega,\nu)}
\geq\Vert M_g\Vert\Vert h\Vert_{L_q(\Omega,\nu)}
$$
Since $h$ is arbitrary $M_{\chi_{\Omega_c}/g}$ is topologically injective.

$ii)\implies i)$ Take arbitrary $h\in L_q(\Omega,\nu)$ and consider $f=M_{\chi_{\Omega_c}/g}(h)$. Then $M_g(f)=M_g(M_{\chi_{\Omega_c}/g}(h))=h$ and $\Vert f\Vert_{L_p(\Omega,\mu)}\leq\Vert M_{\chi_{\Omega_c}/g}\Vert\Vert h\Vert_{L_q(\Omega,\nu)}$. Since $h$ is arbitrary $M_g$ is topologically surjective.
\end{proof}

\subsection{Coisometric operators}

Coisometric operators are particular case of topologically surjective ones, so we can use certain results of previous section. We will use several standards facts about coisometric operators from \cite{HelFA}.

\begin{theorem}\label{CoisomMultOpCharacOnMeasSp} Let $(\Omega,\Sigma,\mu)$ be a $\sigma$-finite measure space, $p,q\in[1,+\infty]$ and $g\in L_0(\Omega,\mu)$, then the following are equivalent

i) $M_g\in\mathcal{B}(L_p(\Omega,\mu),L_q(\Omega,\mu))$ is coisometric 

ii) $M_g$ is an isometric isomorphism 

iii) $|g|=\mu(\Omega)^{1/q-1/p}$, if $p\neq q$ the space $(\Omega,\Sigma,\mu)$ consist of single atom.
\end{theorem}
\begin{proof} Since $M_g$ is coisometric it is topologically injective so from theorem \ref{TopSurMultOpCharacOnMeasSp} we get that $M_g$ is in fact isomorphism. As the consequence it is injective, but injective coisometric operator is an isometric isomorphisms. It is remains to note that every isometric isomorphism is a strict coisometry. Thus we conclude that $M_g$ is coisometric if and only if it is strictly coisometric if and only if it is isometric isomorphism. Now we apply theorem \ref{IsomMultOpCharacOnMeasSp}.
\end{proof}

\begin{proposition}\label{CoisomMultOpCharacBtwnTwoContMeasSp} Let $(\Omega,\Sigma,\nu)$ be a $\sigma$-finite measure space, $p,q\in[1,+\infty]$ and $g,\rho\in L_0(\Omega,\rho\cdot\nu)$ and $\rho$ is non negative, then the following are eqivalent

i) $M_g\in\mathcal{B}(L_p(\Omega,\rho\cdot\nu),L_q(\Omega,\nu))$ is coisometric 

ii) $M_g$ is an isometric isomorphism 

iii) $\rho$ is positive, $|g\cdot \rho^{-1/p}|=\mu(\Omega)^{1/p-1/q}$, if $p\neq q$ the space $(\Omega,\Sigma,\mu)$ consist single atom.
\end{proposition}
\begin{proof} $i)\implies ii)$ Assume $M_g$ is coisometric, then it is topologically surjective. By theorem \ref{TopSurMultOpCharacBtwnTwoContMeasSp} $M_g$ is an isomorphism, hence bijective. It is remains to note that bijective coisometry is an isometric isomorphism.

$ii)\implies i)$ If $M_g$ is an isometric isomorphism, of course, it is coisometry and even more a strict coisometry.

$i)\Longleftrightarrow iii)$ Follows from proposition \ref{IsomMultOpCharacBtwnTwoContMeasSp}.
\end{proof}

\begin{theorem}\label{CoisomMultOpCharacBtwnTwoMeasSp} Let $(\Omega,\Sigma,\mu)$, $(\Omega,\Sigma,\nu)$ be two $\sigma$-finite measure spaces, $p,q\in[1,+\infty]$ and $g\in L_0(\Omega,\mu)$, then the following are equivalent 

i) $M_g\in\mathcal{B}(L_p(\Omega,\mu), L_q(\Omega,\nu))$ is coisometric 

ii) $M_g^{\Omega_c}$ is an isometric isomorphism

iii) $\rho_{\mu,\nu}$ is positive, $|g\cdot\rho_{\mu,\nu}^{-1/p}|_{\Omega_c}|=\mu(\Omega_c)^{1/p-1/q}$, if $p\neq q$ the space $(\Omega,\Sigma,\mu)$ consist of single atom.
\end{theorem}
\begin{proof} $i)\implies ii)$ Since $M_g$ is coisimetric, then from proposition \ref{MultOpDecompDecomp} we know that $M_g^{\Omega_c}$ is also coisometric. From proposition \ref{CoisomMultOpCharacBtwnTwoContMeasSp} we get that $M_g^{\Omega_c}$ is an isometric isomorphism. 

$ii)\implies i)$ Take arbitrary $h\in L_q(\Omega,\nu)$, then there exist $f\in L_p(\Omega_c,\mu|_{\Omega_c})$ such that $M_g^{\Omega_c}(f)=h|_{\Omega_c}$. By proposition \ref{MultOpCharacBtwnTwoSingMeasSp} operator $M_g^{\Omega_s}=0$, so
$$
M_g(\tilde{f})
=\widetilde{M_g^{\Omega_c}(\tilde{f}|_{\Omega_c})}+\widetilde{M_g^{\Omega_s}(\tilde{f}|_{\Omega_s})}
=\widetilde{h|_{\Omega_c}}
$$
Since $\nu(\Omega_s)=0$, then $\Vert h-\widetilde{h|_{\Omega_c}}\Vert_{L_q(\Omega,\nu)}=\Vert h\chi_{\Omega_s}\Vert_{L_q(\Omega,\nu)}=0$ and we conclude $h=\widetilde{h|_{\Omega_c}}$. So we found $\tilde{f}\in L_p(\Omega,\mu)$ such that $M_g(\tilde{f})=h$ and $\Vert \tilde{f}\Vert_{L_p(\Omega,\mu)}=\Vert f\Vert_{L_p(\Omega_c,\mu|_{\Omega_c})}=\Vert h|_{\Omega_c}\Vert_{L_q(\Omega_c,\nu|_{\Omega_c})}\leq\Vert h\Vert_{L_q(\Omega,\nu)}$. Since $h$ is arbitrary then $M_g$ is $1$-topologically surjective. For all $f\in L_p(\Omega,\mu)$ we have
$$
\Vert M_g(f)\Vert_{L_q(\Omega,\nu)}
=\Vert\widetilde{M_g^{\Omega_c}(f|_{\Omega_c})}+\widetilde{M_g^{\Omega_s}(f|_{\Omega_s})}\Vert_{L_q(\Omega,\nu)}
=\Vert\widetilde{M_g^{\Omega_c}(f|_{\Omega_c})}\Vert_{L_q(\Omega,\nu)}
$$
$$
=\Vert M_g^{\Omega_c}(f|_{\Omega_c})\Vert_{L_q(\Omega_c,\nu|_{\Omega_c})}
=\Vert f|_{\Omega_c}\Vert_{L_p(\Omega_c,\mu|_{\Omega_c})}
\leq\Vert f \Vert_{L_p(\Omega,\mu)}
$$
Since $f$ is arbitrary, then $M_g$ is contractive, but it is also $1$-topologically injective. Thus $M_g$ is coisometric.

$i)\Longleftrightarrow iii)$ Follows from proposition \ref{CoisomMultOpCharacBtwnTwoContMeasSp}
\end{proof}

\begin{theorem}\label{CoisomMultOpDescBtwnTwoMeasSp} Let $(\Omega,\Sigma,\mu)$, $(\Omega,\Sigma,\nu)$ be two $\sigma$-finite measure spaces, $p,q\in[1,+\infty]$ and $g\in L_0(\Omega,\mu)$, then the following are equivalent 

i) $M_g\in\mathcal{B}(L_p(\Omega,\mu),L_q(\Omega,\nu))$ is coisometric 

ii) $M_{\chi_{\Omega_c}/g}\in\mathcal{B}(L_q(\Omega,\nu), L_p(\Omega,\mu))$ its right inverse isometric operator.
\end{theorem}
\begin{proof}
$i\implies ii)$ By proposition \ref{MultOpDecompDecomp} operator $M_g^{\Omega_c}$ is coisometric and by proposition \ref{CoisomMultOpCharacBtwnTwoContMeasSp} it is isometric, invertable and $(M_g^{\Omega_c})^{-1}=M_{1/g}^{\Omega_c}$. Then for all $h\in L_q(\Omega,\nu)$ we have
$$
\Vert M_{\chi_{\Omega_c}/g}(h)\Vert_{L_p(\Omega,\mu)}=
\Vert M_{1/g}(h)\chi_{\Omega_c}\Vert_{L_p(\Omega,\mu)}=
\Vert M_{1/g}(h|_{\Omega_c})\Vert_{L_p(\Omega_c,\mu|_{\Omega_c})}=
\Vert M_{1/g}^{\Omega_c}(h|_{\Omega_c})\Vert_{L_p(\Omega_c,\mu|_{\Omega_c})}
$$
$$
=\Vert h|_{\Omega_c}\Vert_{L_q(\Omega_c,\nu|_{\Omega_c})}
\leq\Vert h\Vert_{L_q(\Omega,\nu)}
$$ 
So $M_{\chi_{\Omega_c}/g}$ is contractive. Now note that for all $h\in L_q(\Omega,\nu)$ we have 
$$
M_g(M_{\chi_{\Omega_c}/g}(h))
=M_g(\chi_{\Omega_c}/g\cdot h)
=g\cdot(\chi_{\Omega_c}/g)\cdot  h
=h\cdot\chi_{\Omega_c}
$$
Since $\nu(\Omega\setminus\Omega_c)=0$, then $\chi_{\Omega_c}=\chi_{\Omega}$, so $M_g(M_{\chi_{\Omega_c}/g}(h))=h\cdot\chi_{\Omega_c}=h\cdot\chi_{\Omega}=h$. This means that $M_g$ have right inverse multiplication operator. Take any $h\in L_q(\Omega,\nu)$, then
$$
\Vert M_{\chi_{\Omega_c}/g}(h)\Vert_{L_p(\Omega,\mu)}
\geq\Vert M_g\Vert\Vert M_g(M_{\chi_{\Omega_c}/g}(h))\Vert_{L_q(\Omega,\nu)}
\geq\Vert h\Vert_{L_q(\Omega,\nu)}
$$
Since $h$ is arbitrary $M_{\chi_{\Omega_c}/g}$ is $1$-topologically injective, but it is contractive. Thus $M_g$ is isometric.

$ii)\implies i)$ Take arbitrary $h\in L_q(\Omega,\nu)$ and consider $f=M_{\chi_{\Omega_c}/g}(h)$. Then $M_g(f)=M_g(M_{\chi_{\Omega_c}/g}(h))=h$ and $\Vert f\Vert_{L_p(\Omega,\mu)}\leq\Vert h\Vert_{L_q(\Omega,\nu)}$. Since $h$ is arbitrary $M_g$ is strictly $1$-topologically surjective. Let $f\in L_p(\Omega,\mu)$. By assumption $M_{\chi_{\Omega_c}/g}$ so
$$
\Vert M_g(f)\Vert_{L_q(\Omega,\nu)}
=\Vert M_{\chi_{\Omega_c}/g}(M_g(f))\Vert_{L_p(\Omega,\mu)}
=\Vert f\chi_{\Omega_c}\Vert_{L_p(\Omega,\mu)}
\leq\Vert f\Vert_{L_p(\Omega,\mu)}
$$
Since $f$ is arbitrary $M_g$ is contractive, but it is also strictly $1$-topologically surjective, hence strictly coisometric.
\end{proof}

Note that this proof shows that every coisometric multiplication operator is strictly coisometric. This significantly simplifies considerations of the next section.

\section{Projective, injective and flat $B(\Omega)$-modules in the \\category of $L_p$ spaces}

\subsection{Morphisms of $B(\Omega)$-modules $L_p$}

By $B(\Omega)$ we will denote Banach algebra of bounded measurable functions on measurable space $(\Omega,\Sigma)$ with $\sup$-norm. Obviously for any $b\in B(\Omega)$ and any $f\in L_p(\Omega,\mu)$ we have
$$
\Vert b\cdot f\Vert_{L_p(\Omega,\mu)}
\leq\Vert b\Vert_{B(\Omega)}\Vert f\Vert_{L_p(\Omega,\mu)}
$$
Hence every $L_p$ space is a left/right/two sided Banach $B(\Omega)$-module. Since for the same $f$ and $b$ we have $b\cdot f=f\cdot b$, and algebra $B(\Omega)$ is commutative we can restrict our considerations to the left modules.

By $M(\Omega)$ we will denote Banach space of finite complex valued $\sigma$-additive measures on $\Omega$. By $\mathscr{L}$ we denote full subcategory of left Banach $B(\Omega)$ modules consisting of $L_p(\Omega,\mu)$ spaces for some $\mu\in M(\Omega)$. By $\mathscr{L}_1$ we will denote the  category with the same objects but with contractive morphisms only. By $\mathscr{L}^{\operatorname{op}}$ we will denote the category of the right $B(\Omega)$ modules of the form $L_p(\Omega,\mu)$. In \cite{HelLp} Helemskii gave a complete characterisation of morphisms of $\mathscr{L}$, but only for for locally compact $\Omega$, with Borel $\sigma$-algebra. Careful inspection of his proof shows that this characterization valid for all $\sigma$-finite measure spaces.

Let $p,q\in[1,+\infty]$ and $\mu,\nu\in M(\Omega)$. Denote $\Omega_+=\{\omega\in\Omega_c:\rho_{\nu,\mu}(\omega)>0\}$. Recall that by $\Omega_a$ we denote atomic part of measure space $(\Omega,\Sigma,\mu)$ provided by proposition \ref{DescOfLpSpOnMeasSp}. Of course, we may assume that $\Omega_a\subset\Omega_c$. Introduce the notation
$$
L_{p,q,\mu,\nu}(\Omega)=
\begin{cases}
\{g\in L_0(\Omega,\mu):g\in L_{pq/(p-q)}(\Omega,\rho_{\nu,\mu}^{p/(p-q)}\cdot\mu),\quad g|_{\Omega\setminus\Omega_+}=0\}&\text{if}\quad p>q\\
\{g\in L_0(\Omega,\mu):g\rho_{\nu,\mu}^{1/p}\in L_{\infty}(\Omega,\mu),\quad g|_{\Omega\setminus\Omega_+}=0\}&\text{if}\quad p=q\\
\{g\in L_0(\Omega,\mu):g\rho_{\nu,\mu}^{1/p}\mu^{pq/(p-q)}\in L_{\infty}(\Omega,\mu),\quad g|_{\Omega\setminus\Omega_a}=0\}&\text{if}\quad p<q\\
\end{cases}
$$
$$
\Vert g\Vert_{L_{p,q,\mu,\nu}(\Omega)}=
\begin{cases}
\Vert g\Vert_{L_{pq/(p-q)}(\Omega,\rho^{p/(p-q)}\cdot\mu)}&\text{if}\quad p>q\\
\Vert g\rho_{\nu,\mu}^{1/p}\Vert_{L_{\infty}(\Omega,\mu)}&\text{if}\quad p=q\\
\Vert g\rho_{\nu,\mu}^{1/p}\mu^{pq/(p-q)}\Vert_{L_{\infty}(\Omega,\mu)}&\text{if}\quad p<q\\
\end{cases}
$$
\begin{theorem}[\cite{HelLp}, 4.1]\label{LpModMorphCharac}
Let $p,q\in[1,+\infty]$ and $\mu,\nu\in M(\Omega)$,then there exist isometric isomorphism
$$
\mathcal{I}_{p,q,\mu,\nu}:L_{p,q,\mu,\nu}(\Omega)\to\operatorname{Hom}_{\mathscr{L}}(L_p(\Omega,\mu),L_q(\Omega,\nu)):g\mapsto M_g
$$
\end{theorem}

Simply speaking all morphisms in $\mathscr{L}$ are multiplication operators. Now we need definitions for different types of "good" morphisms from the point of view of Banach homology theory to describe variants of projectivity, injectivity and flatness.

\begin{definition}\label{AdmEpiMorph} Let $\mathscr{C}$ be a category of left Banach modules over algebra $A$. Let $X,Y\in\operatorname{Ob}(\mathscr{C})$, then we say that a morphism $\varphi\in\operatorname{Hom}_{\mathscr{C}}(X,Y)$ is a relatively/metrically/extremelly admissible epimorphism if it is topologically surjective/strictly coisometric/coisometric.
\end{definition}

\begin{definition}\label{AdmMonoMorph} Let $\mathscr{C}$ be a category of left Banach modules over algebra $A$. Let $X,Y\in\operatorname{Ob}(\mathscr{C})$, then we say that morphism $\varphi\in\operatorname{Hom}_{\mathscr{C}}(X,Y)$ is a relatively/metrically/extremelly admissible monomorphism if it is topologically injective/isometric/isometric.
\end{definition}

All these notions are due to Helemskii (\cite{HelBanLocConv},\cite{HelMetFrPoj})Now results of previous section may be reformulated as follows:

1) All relatively/metrically/extremely admissible epimorphisms in $\mathscr{L}$ are retractions and vice versa. 

2) All relatively/metrically/extremely admissible monomorphisms in $\mathscr{L}$ are coretractions and vice versa.

\subsection{Injective $L_p$ modules}

\begin{definition} Let $\mathscr{C}$ be a category of left Banach modules over algebra $A$. We say that $I\in\operatorname{Ob}(\mathscr{C})$ is relatively/metrically/extremely injecive if the functor $\operatorname{Hom}_{\mathscr{C}}(-,I)$ from $\mathscr{C}$/$\mathscr{C}_1$/$\mathscr{C}_1$ to $\mathscr{B}an$ maps relatively/metricaly/extremely admissible monomorphisms to surjective/strictly coisometric/coisometric operators.
\end{definition}

\begin{theorem} Every $B(\Omega)$ module $L_p$ is relatively/metrically/extremely injective in $\mathscr{L}$.
\end{theorem}
\begin{proof}
1) Relative injectivity. Let $I,X,Y\in\operatorname{Ob}(\mathscr{L})$. Take arbitrary $\varphi\in\operatorname{Hom}_{\mathscr{L}}(X,I)$ and relatively admissible monomorphism $i\in\operatorname{Hom}_{\mathscr{L}}(X, Y)$. By theorem \ref{TopInjMultOpDescBtwnTwoMeasSp} we have  topologically surjective $\pi\in\operatorname{Hom}_{\mathscr{L}}(Y, X)$ such that $\pi i=1_{X}$. Then for $\psi=\varphi\pi$ we have $\operatorname{Hom}_{\mathscr{L}}(i, I)(\psi)=\varphi$. Hence $\operatorname{Hom}_{\mathscr{L}}(i, I)$ is surjective. This means that $I$ is relatively injective.

2) Metric/extreme injectivity. Let $I,X,Y\in\operatorname{Ob}(\mathscr{L}_1)$ Take arbitrary $\varphi\in\operatorname{Hom}_{\mathscr{L}_1}(X, I)$ and metrically/extremely admissible monomorphism $i\in\operatorname{Hom}_{\mathscr{L}_1}(X, Y)$. By theorem \ref{IsomMultOpDescBtwnTwoMeasSp} we have coisometric $\pi\in\operatorname{Hom}_{\mathscr{L}_1}(Y, X)$ such that $\pi i=1_{X}$. Then for $\psi=\varphi\pi$ we have $\operatorname{Hom}_{\mathscr{L}_1}(i, I)(\psi)=\varphi$ and what is more $\Vert\psi\Vert=\Vert\varphi\Vert$ because $\Vert\psi\Vert\leq\Vert\varphi\Vert\Vert\pi\Vert=\Vert\varphi\Vert$ and $\Vert\varphi\Vert\leq\Vert\psi\Vert\Vert i\Vert=\Vert\psi\Vert$. Hence $\operatorname{Hom}_{\mathscr{L}_1}(i,I)$ is strictly coisometric and a fortiori coisometric. This means that $I$ is metrically/extremely injective.
\end{proof}

\subsection{Projective $L_p$ modules}

\begin{definition} Let $\mathscr{C}$ be a category of left Banach modules over algebra $A$. We say that $P\in\operatorname{Ob}(\mathscr{C})$ is relatively/metrically/extremely projective if the functor $\operatorname{Hom}_{\mathscr{C}}(P,-)$ from $\mathscr{C}$/$\mathscr{C}_1$/$\mathscr{C}_1$ to $\mathscr{B}an$ maps relatively/metricaly/extremely admissible epimorphisms to surjective/strictly coisometric/coisometric operators.
\end{definition}

\begin{theorem} Every $B(\Omega)$ module $L_p$ is relatively/metrically/extremely projective in $\mathscr{L}$.
\end{theorem}
\begin{proof}
1) Relative projectivity. Let $P,X,Y\in\operatorname{Ob}(\mathscr{L})$. Take arbitrary $\varphi\in\operatorname{Hom}_{\mathscr{L}}(P, X)$ and relatively admissible  epimorphism $\pi\in\operatorname{Hom}_{\mathscr{L}}(Y, X)$. By theorem \ref{TopSurMultOpDescBtwnTwoMeasSp} we have topologically injective $i\in\operatorname{Hom}_{\mathscr{L}}(X, Y)$ such that $\pi i=1_{X}$. Then for $\psi=i\varphi$ we have $\operatorname{Hom}_{\mathscr{L}}(P,\pi)(\psi)=\varphi$. Hence $\operatorname{Hom}_{\mathscr{L}}(P,\pi)$ is surjective. This means that $P$ is relatively projective.

2) Metric/extreme projectivity. Let $P,X,Y\in\operatorname{Ob}(\mathscr{L}_1)$. Take arbitrary $\varphi\in\operatorname{Hom}_{\mathscr{L}_1}(P, X)$ and metrically/extremely admissible epimorphism $\pi\in\operatorname{Hom}_{\mathscr{L}_1}(Y, X)$. By theorem \ref{CoisomMultOpDescBtwnTwoMeasSp} we have isometric $i\in\operatorname{Hom}_{\mathscr{L}_1}(X, Y)$ such that $\pi i=1_{X}$. Then for $\psi=i\varphi$ we have $\operatorname{Hom}_{\mathscr{L}_1}(P,\pi)(\psi)=\varphi$ and what is more $\Vert\psi\Vert=\Vert\varphi\Vert$ because $i$ is isometric. Hence $\operatorname{Hom}_{\mathscr{L}_1}(P,\pi)$ is strictly coisometric and a fortiori coisometric. This means that $P$ is metrically/extremely projective.
\end{proof}

\subsection{Flat $L_p$ modules}

\begin{definition} Let $\mathscr{C}$ be a category of left Banach modules over algebra $A$. We say that $F\in\operatorname{Ob}(\mathscr{C})$ is relatively/metrically/extremely flat if the functor $-\mathop{\operatorname{\otimes}}^A 1_F$ from $\mathscr{C}$/$\mathscr{C}_1$/$\mathscr{C}_1$ to $\mathscr{B}an$ maps relatively/metricaly/extremely admissible monomorphisms in $\mathscr{L}^{\operatorname{op}}$ to topologically injective/isometric/
isometric operators.
\end{definition}

\begin{theorem} Every $B(\Omega)$ module $L_p$ is relatively/metrically/extremely flat in $\mathscr{L}$.
\end{theorem}
\begin{proof}
1) Relative flatness. Let $F,X,Y\in\operatorname{Ob}(\mathscr{L})$. Take arbitrary  relatively admissible monomorphism $i\in\operatorname{Hom}_{\mathscr{L}^{\operatorname{op}}}(X, Y)$. By theorem \ref{TopInjMultOpDescBtwnTwoMeasSp} we have topologically surjective $\pi\in\operatorname{Hom}_{\mathscr{L}^{\operatorname{op}}}(Y, X)$ such that $\pi i=1_{X}$. Then for arbitrary $u\in F\mathop{\operatorname{\otimes}}^{B(\Omega)} X$
we have
$$
\Vert(1_F \mathop{\operatorname{\otimes}}^{B(\Omega)} \pi)\Vert\Vert(1_F \mathop{\operatorname{\otimes}}^{B(\Omega)} i)(u)\Vert_{F\mathop{\operatorname{\otimes}}^{B(\Omega)} Y}
\geq 
\Vert(1_F \mathop{\operatorname{\otimes}}^{B(\Omega)} \pi)(1_F \mathop{\operatorname{\otimes}}^{B(\Omega)} i)(u)\Vert_{F\mathop{\operatorname{\otimes}}^{B(\Omega)} X}
=
\Vert(1_F \mathop{\operatorname{\otimes}}^{B(\Omega)} \pi i)(u)\Vert_{F\mathop{\operatorname{\otimes}}^{B(\Omega)} X}
$$
$$
=\Vert(1_F \mathop{\operatorname{\otimes}}^{B(\Omega)} 1_X)(u)\Vert_{F\mathop{\operatorname{\otimes}}^{B(\Omega)} X}
=\Vert u\Vert_{F\mathop{\operatorname{\otimes}}^{B(\Omega)} X}
$$
Also note that $\Vert(1_F \mathop{\operatorname{\otimes}}^{B(\Omega)} \pi)\Vert\leq\Vert 1_F\Vert\Vert\pi\Vert$, hence
$$
\Vert(1_F \mathop{\operatorname{\otimes}}^{B(\Omega)} i)(u)\Vert_{F\mathop{\operatorname{\otimes}}^{B(\Omega)} Y}
\geq
\Vert\pi\Vert^{-1}\Vert u\Vert_{F\mathop{\operatorname{\otimes}}^{B(\Omega)} X}
$$
Thus $1_F \mathop{\operatorname{\otimes}}^{B(\Omega)} i$ is topologically injective, so $F$ is relatively flat.

2) Metric/extreme projectivity. Let $F,X,Y\in\operatorname{Ob}(\mathscr{L})$. Take arbitrary metrically/extremely admissible monomorphism $i\in\operatorname{Hom}_{\mathscr{L}^{\operatorname{op}}}(X, Y)$. By theorem \ref{TopInjMultOpDescBtwnTwoMeasSp} we have coisometric $\pi\in\operatorname{Hom}_{\mathscr{L}^{\operatorname{op}}}(Y, X)$ such that $\pi i=1_{X}$. Fix $u\in F\mathop{\operatorname{\otimes}}^{B(\Omega)} X$. Since $\pi$ is cosiometric then from previous paragraph we get
$$
\Vert(1_F \mathop{\operatorname{\otimes}}^{B(\Omega)} i)(u)\Vert_{F\mathop{\operatorname{\otimes}}^{B(\Omega)} Y}
\geq
\Vert\pi\Vert^{-1}\Vert u\Vert_{F\mathop{\operatorname{\otimes}}^{B(\Omega)} X}
\geq
\Vert u\Vert_{F\mathop{\operatorname{\otimes}}^{B(\Omega)} X}
$$
On the other hand for the same $u$ we have
$$
\Vert(1_F \mathop{\operatorname{\otimes}}^{B(\Omega)} i)(u)\Vert_{F\mathop{\operatorname{\otimes}}^{B(\Omega)} Y}
\leq
\Vert 1_F \mathop{\operatorname{\otimes}}^{B(\Omega)} i\Vert\Vert u\Vert_{F\mathop{\operatorname{\otimes}}^{B(\Omega)} X}
\leq
\Vert 1_F\Vert\Vert i\Vert
\Vert u\Vert_{F\mathop{\operatorname{\otimes}}^{B(\Omega)} X}
=
\Vert u\Vert_{F\mathop{\operatorname{\otimes}}^{B(\Omega)} X}
$$
From these inequalities it follows that $1_F \mathop{\operatorname{\otimes}}^{B(\Omega)} i$ is isometric.
\end{proof}

Norbert Nemesh, Faculty of Mechanics and Mathematics, Moscow State University, Moscow 119991 Russia

\textit{E-mail address:} nemeshnorbert@yandex.ru

\end{document}